\documentclass[11pt]{article}

\usepackage{amsfonts,amsmath,amsthm,amssymb}
\usepackage{bm}
\usepackage[colorlinks, citecolor=blue, urlcolor=blue, backref=page]{hyperref}
\usepackage[round]{natbib}
\usepackage[margin=1.2in]{geometry}
\usepackage[utf8]{inputenc}
\usepackage[T1]{fontenc}    

\usepackage{multirow}
\usepackage{pdflscape}
\usepackage{tabularx}

\usepackage{enumitem}

\newtheorem{theorem}{Theorem}[section]
\newtheorem{corollary}[theorem]{Corollary}
\newtheorem{lemma}[theorem]{Lemma}

\newtheorem{example}[theorem]{Example}
\theoremstyle{definition}
\newtheorem{definition}[theorem]{Definition}

\newtheorem{remark}[theorem]{\textbf{Remark}}
\numberwithin{equation}{section}

\newcommand{\e}{{\rm e}}

\newcommand{\E}{{\mathsf E}}

\renewcommand{\P}{{\mathsf P}}
\newcommand{\Q}{{\mathsf Q}}
\newcommand{\R}{{\mathbb R}}

\newcommand{\N}{{\mathbb N}}
\newcommand{\U}{{\mathbb U}}

\newcommand{\Ccal}{{\mathcal C}}

\newcommand{\Ecal}{{\mathcal E}}
\newcommand{\Fcal}{{\mathcal F}}
\newcommand{\Gcal}{{\mathcal G}}

\newcommand{\Lcal}{{\mathcal L}}

\newcommand{\Mcal}{{\mathcal M}}

\newcommand{\Pcal}{{\mathcal P}}
\newcommand{\Qcal}{{\mathcal Q}}

\newcommand{\Xcal}{{\mathcal X}}

\newcommand{\fdot}{{\,\cdot\,}}

\DeclareMathOperator{\conv}{conv}

\DeclareMathOperator{\supp}{supp}
\DeclareMathOperator{\dom}{dom}
\DeclareMathOperator{\cone}{cone}
\DeclareMathOperator{\ri}{ri}

\DeclareMathOperator{\range}{range}

\begin{document}

\title{Testing hypotheses generated by constraints}
\author{
Martin Larsson\footnote{Department of Mathematical Sciences, Carnegie Mellon University, \texttt{larsson@cmu.edu}} \and
Aaditya Ramdas\footnote{Departments of Statistics and ML, Carnegie Mellon University, \texttt{aramdas@cmu.edu}} \and
Johannes Ruf\footnote{Department of Mathematics, London School of Economics, \texttt{j.ruf@lse.ac.uk}} \\[2ex]
}
\maketitle

\begin{abstract}
E-variables are nonnegative random variables with expected value at most one under any distribution from a given null hypothesis. 
Every nonasymptotically valid test can be obtained by thresholding some e-variable. As such, e-variables arise naturally in applications in statistics and operations research, and a key open problem is to characterize their form. 
We provide a complete solution to this problem for hypotheses generated by constraints---a broad and natural framework that encompasses many hypothesis classes occurring in practice. 
Our main result is an abstract representation theorem that describes all e-variables for any hypothesis defined by an arbitrary collection of measurable constraints.
We instantiate this general theory for three important classes: hypotheses generated by finitely many constraints, one-sided sub-$\psi$ distributions (including sub-Gaussian distributions), and distributions constrained by group symmetries. In each case, we explicitly characterize all e-variables as well as all admissible e-variables. 
Numerous examples are treated, including constraints on moments, quantiles, and conditional value-at-risk (CVaR).
Building on these, we prove existence and uniqueness of optimal e-variables under a large class of expected utility-based objective functions used for optimal decision making, in particular covering all criteria studied in the e-variable literature to date. \end{abstract}

\section{Introduction}

Fix a measurable space $\Xcal$ and let $\Mcal_1$ denote the set of all probability measures on $\Xcal$. Suppose we observe a random datum $X$ with values in $\Xcal$, and consider the null hypothesis that the distribution of $X$ belongs to some set of probability measures $\Pcal \subset \Mcal_1$. 
An e-variable for $\Pcal$ is a nonnegative (possibly infinite) random variable whose expected value under every distribution in $\Pcal$ is at most one. 
The set of all e-variables for  $\Pcal$ is denoted by $\Ecal$:
\[
\Ecal = \left\{\text{all measurable } h \colon \Xcal \to [0,\infty] \text{ such that } \int_\Xcal h d\mu \le 1 \text{ for all } \mu \in \Pcal\right\}.
\]
E-variables have recently been recognized as fundamental objects in a variety of hypothesis testing and inference problems. A rapidly growing body of work uses e-variables as the basis for solving a wide range of problems in statistics and operations research, such as multiple testing, A/B testing, and sequential anytime-valid inference, just to mention a few. Some recent papers include~\cite{wasserman2020universal,vovk2021values,shafer2021testing,grunwald2020safe}
; several other key references appear later in this paper.

To see why e-variables are fundamentally connected to hypothesis testing, observe that every e-variable for $\Pcal$ yields a nonasymptotic level-$\alpha$ test for $\Pcal$: we reject the null when the e-variable exceeds $1/\alpha$. Markov's inequality implies that the type-I error of such a test is at most $\alpha$. Conversely, it is known that every level-$\alpha$ test for $\Pcal$ can be recovered by thresholding some e-variable at $1/\alpha$. 
Consequently, a description of all e-variables leads to a description of all valid tests.
Further, it is desirable to identify e-variables that lead to \emph{powerful} tests. Standard notions of power for e-variables are based on expected utility under an alternative hypothesis \citep{shafer2021testing,grunwald2020safe,larsson2024numeraire}, and finding powerful e-variables amounts to solving optimization problems over the set $\Ecal$. In addition to their importance for statistical analysis, solutions to these problems serve as input for optimal decision making.

For these reasons, it is of interest to characterize $\Ecal$. Indeed, a characterization of $\Ecal$ is effectively a characterization of the set of all tests for $\Pcal$. Furthermore, a low-dimensional parameterization of $\Ecal$ simplifies the task of finding optimal e-variables.

In this paper we characterize all e-variables for sets $\Pcal$ that are described by constraints. 
The constraints we consider can be viewed as generalized moment constraints, meaning that certain functions of the observed data must have nonpositive expectation.
This is a very natural class, especially in nonparametric settings. For example, it is common to test whether some functional, such as the mean or variance, lies in some range; this is a constraint. It is also common to specify classes of distributions whose moments or supports are restricted in some way; these are also constraints.
Thus, the classes $\Pcal$ considered in this paper are quite general, and special cases of such ``constrained hypotheses'' have been frequently considered in the literature. 
Our work characterizes all possible e-variables for such classes, without restrictions on the constraints (there could be uncountably many, they can be discontinuous functions, etc.) and without restrictions on the underlying measure space (we do not require any topological properties such as compactness, closedness, finite dimensionality, etc.).
In a nutshell, our main results show that every e-variable can be expressed in terms of conic combinations of the constraint functions.

While we will cite papers that study special cases of $\Pcal$ in later sections of the paper that instantiate our general results, we note that the only general attempt to characterize e-variables for constrained hypotheses appears in a recent work by~\cite{clerico2024optimalevaluetestingproperly}. Among our contributions are greatly generalized versions of the results in that paper.
For example, two of their most important restrictions are that $\Xcal$ is a closed subset of $\R^n$ and that constraint functions are continuous, both of which we do not need. Their paper has other minor restrictions, for example they disallow e-variables from taking the value $+\infty$, and they need $0$ to be in the relative interior of the convex hull of the constraints; we avoid these. Their Lemma~4 shows that these restrictions imply that $\Pcal$ must include distributions that put mass on arbitrary points in $\Xcal$; our Theorem~\ref{T_finite_constraint_set_reference_measure} allows imposing a reference measure like Lebesgue. Finally, and critically, they handle only finitely many constraints. We would also like to highlight the nonasymptotic nature of our work. This is in contrast to, for example, \citet{eb0770a2-159e-3683-976d-c3be11a128bb}, who design asymptotically valid tests for finitely constrained hypotheses.

Lastly, although we frame our results as being about a single observation from the sample space $\Xcal$, we haste to emphasize two things. First, this ``single observation'' could be a finite or infinite sample, say if $\Xcal = \R^n$ or $\Xcal = \R^\infty$, and our results still apply. Second, given an e-variable $h(x)$ for $\Pcal$, the product $\prod_{i=1}^n h(X_i)$ is an e-variable for any i.i.d.\ sequence $X_1,\ldots,X_n$ sampled from a distribution in $\Pcal$. More generally, it is an e-variable provided that for all $i$, the conditional distribution of $X_{i+1}$ given $X_1,\ldots,X_i$ belongs to $\Pcal$. Thus our results immediately yield explicit e-variables in these settings as well, although we do not characterize \emph{all} e-variables in this case. This and other challenging open problems are discussed in Section~\ref{S_summary}.

\paragraph{Paper outline.} Section~\ref{S_gen_hyp} formally defines hypotheses generated by constraints, introduces the mathematical setting needed to analyze them, and presents the main result of this paper, Theorem~\ref{T_abs_rep}, which gives an abstract description of the set of all e-variables for any hypothesis generated by constraints. An important role is played by the theory of dual pairs of vector spaces, and in particular the notion of weak closure; we review these in the appendix.

Sections~\ref{S_fin_gen_hyp}--\ref{S_group_symmetry} instantiate the abstract theorem for three particular hypothesis classes of practical interest. In Section~\ref{S_fin_gen_hyp}, hypotheses generated by finitely many constraints are considered. The main results are Theorem~\ref{T_finitely_generated} and Corollary~\ref{C_evar_max_finite}, and several special cases are discussed thereafter. 
Section~\ref{S_sub_psi} tackles an important and general class of distributions generated by uncountably many constraints, that of one-sided sub-$\psi$ distributions, which includes the well-studied sub-Gaussian case. The main result is Theorem~\ref{T_sub_psi_rep}.
Section~\ref{S_group_symmetry} is devoted to another nontrivial and general class of distributions, those that are constrained to remain invariant under a group of symmetries (such as  exchangeable distributions). The main results are Theorems~\ref{T_symmetric_rep} and~\ref{T_group_inv_constr_set}.

Up to this point in the paper, the focus is on obtaining explicit descriptions of the set of \emph{all} e-variables for a given hypothesis. However, in statistical applications one wishes to work with \emph{admissible} e-variables whenever possible. This is the topic of Section~\ref{S_admissibility}, where \emph{minimal complete classes} of (admissible) e-variables are introduced. It is shown that such classes do not exist in general, but do exist in the setting of this paper under mild additional assumptions. The main result is Theorem~\ref{T_minimal_complete_class}.

Section~\ref{S:optimal} tackles the problem of finding \emph{optimal} e-variables for finitely generated and sub-$\psi$  hypotheses. For a general class of objective functions, existence of optimizers is established in Theorems~\ref{T_opt_evar_finite} and~\ref{T_optimal_sub_psi}, and a simple uniqueness criterion is provided in Theorem~\ref{T_optimal_uniqueness}. These results cover all objective functions studied in the e-variable literature to date, as well as criteria such as the continuous variational preferences of \citet{MR2268407}. The results in this section depend crucially on the representation theorems developed earlier.

The last part of the paper contains further discussion and amplifications. Section~\ref{S_unions} studies hypotheses that are unions of hypotheses generated by constraints. 
A nontrivial example involving conditional value-at-risk (CVaR) illustrates its utility. 
Section~\ref{S_relaxed_integrability} considers hypotheses with a relaxed integrability condition. It is shown that for finitely generated hypotheses, this relaxation makes no difference to the set of associated e-variables, but leaves an open question in the infinite case.
Section~\ref{S_summary} concludes with a discussion of further open problems. Appendix~\ref{S_functional_analysis} contains some basic results from topology and functional analysis that are used in the paper, for example providing background for our convergence results that rely on \emph{nets} instead of sequences, and reviewing a key bipolar theorem that underlies our main results.

\paragraph{Notation.} 
We denote by $\Lcal$ and $\Mcal$ the space of all real-valued measurable functions and finite signed measures, respectively, on our measurable space $\Xcal$. We write $|\mu|$ for the total variation measure of any $\mu \in \Mcal$. Given a subset $\Pcal \subset \Mcal_1$ (the probability measures on $\Xcal$), we say that a measurable set $A \subset \Xcal$ is $\Pcal$-negligible if $\mu(A) = 0$ for all $\mu \in \Pcal$. A pointwise property of a function $f \in \Lcal$  holds $\Pcal$-quasi-surely, abbreviated $\Pcal$-q.s., if the set where it fails is $\Pcal$-negligible.
A function $f \in \Lcal$ is called a $\Pcal$-version of another function $g \in \Lcal$ if $f = g$, $\Pcal$-q.s.
For subsets $A, B$ of a vector space, we write $A - B = \{a - b \colon a \in A, b \in B\}$. The set of natural numbers is $\N=\{1,2,\ldots\}$.



\section{Hypotheses generated by constraints} \label{S_gen_hyp}

\begin{definition} \label{D_constrained_hypotheses}
A \emph{constraint set} is any nonempty set of functions $\Phi \subset \Lcal$. The elements of $\Phi$ are called \emph{constraint functions}. The \emph{hypothesis generated by $\Phi$} is the (possibly empty) set $\Pcal$ of probability measures given by
\begin{equation}\label{eq_hypothesis_integrable}
\Pcal = \left\{\mu \in \Mcal_1 \colon \int_\Xcal |f| d\mu < \infty \text{ and } \int_\Xcal f d\mu \le 0 \text{ for all } f \in \Phi\right\}.
\end{equation}
\end{definition}

Although the hypothesis $\Pcal$ is defined through inequality constraints, it is easy to encode equality constraints by letting $\Phi$ contain both $f$ and $-f$.
Given a constraint set $\Phi$ and the hypothesis $\Pcal$ that it generates, we define the vector spaces
\begin{align*}
\Lcal^\Phi &= \left\{f \in \Lcal \colon \int_\Xcal |f| d\mu < \infty \text{ for all } \mu \in \Pcal\right\}, \\
\Mcal^\Phi &= \left\{\mu \in \Mcal \colon \int_\Xcal |f| d|\mu| < \infty \text{ for all } f \in \Lcal^\Phi\right\}.
\end{align*}
These spaces serve as a useful arena for our theory because $\Lcal^\Phi$ contains all bounded measurable functions, all constraint functions, and all (finite) e-variables. Moreover, $\Mcal^\Phi$ contains the hypothesis $\Pcal$ and all Dirac measures, and its elements integrate all the functions in $\Lcal^\Phi$ by construction. It is also convenient to introduce the \emph{quasi-surely positive cone}, 
\[
\Lcal^\Phi_p = \{f \in \Lcal^\Phi \colon f \ge 0, \text{ $\Pcal$-q.s.}\},
\]
with subscript ``$p$'' for ``positive'', as well as the set of functions in $\Lcal^\Phi$ that are nonnegative \emph{everywhere},
\[
\Lcal^\Phi_+ = \{f \in \Lcal^\Phi \colon f \ge 0\}.
\]
These two sets are different in general.
Finally, we define
\begin{equation} \label{eq_Ccal_new}
\Ccal = \cone(\Phi) - \Lcal^\Phi_p.
\end{equation}
This is the convex cone consisting of all functions that are quasi-surely dominated by a conic combination of constraint functions. In symbols, $\Ccal$ consists of all $f = g - h$ with $g \in \cone(\Phi)$ and $h \in \Lcal^\Phi_p$, or equivalently, $f \le g$, $\Pcal$-q.s.
Note that $\cone(\Phi)$ denotes the set of all (finite) conic combinations of elements of $\Phi$, and as such is convex.

The elements of $\Ccal$ have nonpositive expectation under every measure in $\Pcal$. The following result shows that the \emph{weak closure} of $\Ccal$ actually consists of \emph{all} functions in $\Lcal^\Phi$ with this property. This immediately leads to a description of the set of all e-variables for $\Pcal$. Here the weak closure refers to the topology $\sigma(\Lcal^\Phi, \Mcal^\Phi)$ induced by the dual pairing $\langle f, \mu\rangle = \int_\Xcal f d\mu$; see Appendix~\ref{S_dual_pairs}. 

\begin{theorem} \label{T_abs_rep}
\begin{enumerate}
\item\label{T_abs_rep_1} A function $f \in \Lcal^\Phi$ satisfies $\int_\Xcal f d\mu \le 0$ for all $\mu \in \Pcal$ if and only if $f$ belongs to $\overline\Ccal$, the weak closure of $\Ccal$. 
\item\label{T_abs_rep_2} In particular, the set $\Ecal$ of all e-variables for $\Pcal$ consists precisely of those $[0,\infty]$-valued measurable functions that are $\Pcal$-q.s.\ equal to $1 + f$ for some $f \in \overline\Ccal$.
\end{enumerate}
\end{theorem}

\begin{proof}
Let us first confirm that the bilinear form $\langle f, \mu\rangle = \int_\Xcal f d\mu$ separates points. Indeed, if $\mu \in \Mcal^\Phi$ is fixed and $\langle f, \mu\rangle = 0$ for all $f \in \Lcal^\Phi$, then by taking $f = \bm1_A$ for any measurable set $A \subset \Xcal$, we see that $\mu = 0$. If instead $f \in \Lcal^\Phi$ is fixed and $\langle f, \mu\rangle = 0$ for all $\mu \in \Mcal^\Phi$, we may take $\mu = \delta_x$ for any $x \in \Xcal$ to see that $f = 0$. With this out of the way we may proceed with the proof of the theorem.

\ref{T_abs_rep_1}: The bipolar theorem states that $\Ccal^{\circ\circ} = \overline\Ccal$; see Theorem~\ref{T_bipolar_theorem}. This is actually the desired conclusion because, as we show next,
\[
\Ccal^{\circ\circ} = \left\{f\in \Lcal^\Phi \colon \int_\Xcal f d\mu \le 0 \text{ for all } \mu \in \Pcal \right\}.
\]
This representation of $\Ccal^{\circ\circ}$ follows directly from the identity $\Ccal^\circ = \R_+ \Pcal$, which we now prove. (Here $\R_+\Pcal$ refers to the set of nonnegative multiples of elements of $\Pcal$.) For the forward inclusion, consider an element $\mu \in \Ccal^\circ$. Since $-\bm1_A \in \Ccal$ for every measurable set $A$ we have $\mu \in \Mcal_+$, and since $\Phi \subset \Ccal$ we have $\int_\Xcal f d\mu \le 0$ for all $f \in \Phi$. Thus $\mu$ is a nonnegative multiple of an element of $\Pcal$. Conversely, if $\mu$ is a nonnegative multiple of an element of $\Pcal$, then for any function $f \in \Ccal$, say $f = g - h$ with $g \in \cone(\Phi)$ and $h \in \Lcal^\Phi_p$, we have $\int_\Xcal f d\mu \le \int_\Xcal g d\mu \le 0$, and hence $\mu \in \Ccal^\circ$. Thus $\Ccal^\circ = \R_+ \Pcal$, and the proof is complete.

\ref{T_abs_rep_2}: Since every e-variable is $\Pcal$-q.s.\ equal to a finite e-variable, and since $\Lcal^\Phi$ contains all finite e-variables, the claim is immediate from \ref{T_abs_rep_1}.
\end{proof}

\begin{remark}
Although $\Phi$ is assumed to be nonempty, $\Pcal$ may be empty, and Theorem~\ref{T_abs_rep} still applies. In this case, $\Lcal^\Phi$ is the set $\Lcal$ of all measurable functions $f$. The condition $\int_\Xcal f d\mu \le 0$ for all $\mu \in \Pcal$ is vacuously satisfied by \emph{every} such $f$, so the theorem states that $\overline\Ccal$ equals all of $\Lcal$. This can also be seen directly: if $\Pcal = \emptyset$ then every measurable set $A \subset \Xcal$ is negligible, thus $\Lcal^\Phi_p = \Lcal^\Phi = \Lcal$, and $\Ccal = \cone(\Phi) - \Lcal^\Phi_p = \Lcal$ (recall that $\Phi$ is nonempty).
\end{remark}

\begin{remark} \label{R_ref_measure}
A hypothesis generated by constraints will not admit a reference measure in general. However, if a desired reference measure $\bar\mu$ is given, one can ensure that every $\mu \in \Pcal$ satisfies $\mu \ll \bar\mu$ simply by augmenting $\Phi$ with all functions $\bm1_A$ where $A$ is a $\bar\mu$-nullset. At the end of Section~\ref{S_fin_gen_hyp} we discuss this construction in the context of finitely generated hypotheses, see especially Theorem~\ref{T_finite_constraint_set_reference_measure}.
\end{remark}

For later use we record the following basic property of the quasi-surely positive cone.

\begin{lemma}\label{L_Lphip_closed}
$\Lcal^\Phi_p$ is weakly closed.
\end{lemma}

\begin{proof}
We claim that
\begin{equation} \label{eq_L_phi_p_closed}
\Lcal^\Phi_p = \left\{f \in \Lcal^\Phi \colon \int_\Xcal f d\mu \ge 0 \text{ for all } \mu \in \Mcal^\Phi_+ \text{ with } \mu \ll \Pcal\right\},
\end{equation}
where $\Mcal^\Phi_+$ is the set of nonnegative elements of $\Mcal^\Phi$, and $\mu \ll \Pcal$ means that $\mu(A) = 0$ for every $\Pcal$-negligible set $A$. The forward inclusion ``$\subset$'' is clear. For the reverse inclusion ``$\supset$'', consider some $f \notin \Lcal^\Phi_p$. Then the set $A_0 = \{f < 0\}$ is not $\Pcal$-negligible, and hence $\nu(A_0) > 0$ for some $\nu \in \Pcal$. Define $\mu = \nu(\fdot \cap A_0)$. Then $\mu$ belongs to $\Mcal^\Phi_+$ and $\mu \ll \Pcal$ (indeed, $\mu \ll \nu$). But $\int_\Xcal f d\mu < 0$, so $f$ does not belong to the right-hand side of \eqref{eq_L_phi_p_closed}. This establishes \eqref{eq_L_phi_p_closed} and shows that $\Lcal^\Phi_p$ is an intersection of sets of the form $\{f \in \Lcal^\Phi\colon \int_\Xcal f d\mu \ge 0\}$. Since $f \mapsto \int_\Xcal f d\mu$ is weakly continuous, all these sets are weakly closed, and thus so is $\Lcal^\Phi_p$.
\end{proof}

\section{Finitely generated hypotheses} \label{S_fin_gen_hyp}

Consider a finite nonempty constraint set,
\[
\Phi = \{g_1, \ldots, g_d\},
\]
and let $\Pcal$ be the hypothesis generated by $\Phi$. The constraint functions must be real-valued and measurable, but can otherwise be completely arbitrary.

\begin{theorem} \label{T_finitely_generated}
A function $f \in \Lcal^\Phi$ satisfies $\int_\Xcal f d\mu \le 0$ for all $\mu \in \Pcal$ if and only if
\[
f \le \sum_{i=1}^d \pi_i g_i, \quad \text{$\Pcal$-q.s.}
\]
for some $\pi = (\pi_1,\ldots,\pi_d) \in \R^d_+$. In particular, the set $\Ecal$ of all e-variables for $\Pcal$ consists precisely of those $[0,\infty]$-valued measurable functions which are $\Pcal$-q.s.~dominated by
\[
1 + \sum_{i=1}^d \pi_i g_i
\]
for some $\pi = (\pi_1,\ldots,\pi_d) \in \R^d_+$.
\end{theorem}

Before giving the proof, we introduce some terminology.  The \emph{support} of a vector $\rho \in \R^d_+$ is the set $\supp(\rho) = \{i \colon \rho_i > 0\}$. An index set $I \subset \{1,\ldots,d\}$ is called \emph{redundant} if there exists some nonzero $\rho \in \R^d_+$ with $\supp(\rho) \subset I$ such that $\sum_{i=1}^d \rho_i g_i = 0$, $\Pcal$-q.s. Note in particular that the empty set $I = \emptyset$ is \emph{not} redundant. Note also that the set of vectors whose supports are not redundant,
\begin{equation} \label{eq_nonredundant_K}
K = \left\{ \pi \in \R^d_+ \colon  \supp(\pi) \text{ is not redundant} \right\},
\end{equation}
is closed. Indeed, if $\pi_n \in K$ converges to some $\pi \in \R^d_+$, then $\supp(\pi) \subset \supp(\pi_n)$ for all sufficiently large $n$. Since the latter are not redundant, neither is the former, so $\pi \in K$. Finally, recall the set $\Ccal = \cone(\Phi) - \Lcal^\Phi_p$ introduced in \eqref{eq_Ccal_new}.

\begin{lemma} \label{L_not_redundant}
Every $g' \in \cone(\Phi)$ has a $\Pcal$-version $g = \sum_{i = 1}^d \pi_i g_i$ for some $\pi \in K$. Thus every $f \in \Ccal$ is of the form $f = g - h$ with $g = \sum_{i=1}^d \pi_i g_i$ for some $\pi \in K$ and $h \in \Lcal^\Phi_p$.
\end{lemma}

\begin{proof}
Consider any $g' = \sum_{i=1}^d \pi'_i g_i \in \cone(\Phi)$ with $\pi' \in \R^d_+$. To see that $g'$ has a $\Pcal$-version $g = \sum_{i = 1}^d \pi_i g_i$ for some $\pi \in K$, suppose $\pi'$ does not already belong to $K$, meaning that $\supp(\pi')$ is redundant. Let $\rho \in \R^d_+$ be as in the definition of redundant. Then there exists $\varepsilon > 0$ such that $\pi'' = \pi' - \varepsilon \rho$ belongs to $\R^d_+$ and satisfies $\supp(\pi'') \subsetneq \supp(\pi')$. Note that $\sum_i \pi''_i g_i = \sum_i \pi'_i g_i$, $\Pcal$-q.s. If $\supp(\pi'')$ is not redundant, we take $\pi = \pi''$. Otherwise we repeat the process, each time reducing the size of the support. Since the empty set is not redundant, we must eventually reach a representation in terms of a vector $\pi$ whose support is not redundant. This proves the first statement of the lemma.
For the second statement, consider any $f = g' - h' \in \Ccal$ with $g' \in \cone(\Phi)$ and $h' \in \Lcal^\Phi_p$. Let $g = \sum_{i = 1}^d \pi_i g_i$ be a $\Pcal$-version of $g'$ with $\pi \in K$ and note that $f = g - h$ where $h = h' + g - g'$ still belongs to $\Lcal^\Phi_p$.
\end{proof}

\begin{proof}[Proof of Theorem~\ref{T_finitely_generated}]
We will show that the set $\Ccal$ in \eqref{eq_Ccal_new} is already weakly closed; the result then follows from Theorem~\ref{T_abs_rep}. We will prove closedness by showing that the limit of any \emph{convergent net} in $\Ccal$ is again an element of $\Ccal$. Nets are generalizations of sequences, and are required for checking closedness in certain topological spaces. For the benefit of readers who do not work with nets regularly, we review the basic definitions and properties in Appendix~\ref{S_nets}. Readers who are not familiar with nets may replace `net' with `sequence' and `$\alpha$' with `$n$' everywhere below without losing any essential ideas. This modification of the proof would show that $\Ccal$ is \emph{sequentially closed}, but this does not imply closedness in general. For this reason the actual proof uses nets. We now turn to the details.

Consider a net $(f_\alpha)$ in $\Ccal$ that converges weakly to some $f \in \Lcal^\Phi$. We must show that $f \in \Ccal$. Thanks to Lemma~\ref{L_not_redundant}, for each $\alpha$ we have $f_\alpha = g_\alpha - h_\alpha$, where $g_\alpha = \sum_{i=1}^d \pi_{\alpha,i} g_i$ for some $\pi_\alpha \in K$ and $h_\alpha \in \Lcal^\Phi_p$. We claim that the real-valued net $\|\pi_\alpha\| = \pi_{\alpha,1} + \cdots + \pi_{\alpha,d}$ cannot converge to infinity. Assume for contradiction that it does, and write
\begin{equation} \label{eq_fin_gen_pf_1}
\sum_{i=1}^d \frac{\pi_{\alpha,i}}{1+\|\pi_\alpha\|} g_i - \frac{f_\alpha}{1+\|\pi_\alpha\|} = \frac{h_\alpha}{1+\|\pi_\alpha\|}.
\end{equation}
Since $\pi_\alpha / (1+ \|\pi_\alpha\|)$ is a bounded net in $\R^d_+$, we may pass to a subnet and assume that it converges to some limit $\rho \in \R^d_+$, which then satisfies $\|\rho\| = 1$. Each $\pi_\alpha / (1+ \|\pi_\alpha\|)$ belongs to the closed set $K$, so $\rho$ does too. Since the vector space operations are weakly continuous, the first term on the left-hand side of \eqref{eq_fin_gen_pf_1} converges weakly to $\sum_{i=1}^d \rho_i g_i$. Next, for any $\mu \in \Mcal^\Phi$ we have $\langle f_\alpha, \mu\rangle \to \langle f, \mu\rangle$ and hence
\[
\left\langle \frac{f_\alpha}{1+\|\pi_\alpha\|}, \mu \right\rangle = \frac{1}{1+\|\pi_\alpha\|} \langle f_\alpha, \mu \rangle \to 0.
\]
Thus the second term on the left-hand side of \eqref{eq_fin_gen_pf_1} converges weakly to zero. Overall, the left-hand side converges to $\sum_{i=1}^d \rho_i g_i$, and we conclude using Lemma~\ref{L_Lphip_closed} that this quantity is nonnegative, $\Pcal$-q.s.
On the other hand, we have
\[
\int_\Xcal \sum_{i = 1}^d \rho_i g_i d \mu = \sum_{i = 1}^d \rho_i \int_\Xcal g_i d \mu \le 0
\]
for every $\mu \in \Pcal$. Thus $\sum_{i=1}^d \rho_i g_i = 0$, $\Pcal$-q.s., which contradicts the fact that $\supp(\rho)$ is not redundant. We conclude that $\|\pi_\alpha\|$ cannot converge to infinity.

Since $\|\pi_\alpha\|$ does not converge to infinity, it admits a convergent subnet, which we again denote by $\pi_\alpha$. Denote the limit by $\pi \in \R^d_+$. It follows that $g_\alpha$ converges to $g = \sum_{i=1}^d \pi_i g_i$, and then that $h_\alpha = g_\alpha - f_\alpha$ converges to $h = g - f$. The latter belongs to $\Lcal^\Phi_p$ since this set is weakly closed thanks to Lemma~\ref{L_Lphip_closed}. We conclude that $f = g - h \in \Ccal$, showing that $\Ccal$ is weekly closed, as required.
\end{proof}

An immediate consequence of Theorem~\ref{T_finitely_generated} is the following generalization of a result of \citet[Theorem~1]{clerico2024optimalevaluetestingproperly}. Here an e-variable $h$ is called \emph{admissible} if whenever another e-variable $h'$ satisfies $h' \ge h$, $\Pcal$-q.s., we actually have $h' = h$, $\Pcal$-q.s. (See also Section~\ref{S_admissibility} for further discussion of admissibility.)

\begin{corollary}\label{C_evar_max_finite}
Every admissible e-variable is $\Pcal$-q.s.\ equal to
\[
1 + \sum_{i=1}^d \pi_i g_i
\]
for some $\pi$ in the set
\[
\Pi^\Phi = \left\{\pi \in \R^d_+ \colon 1 + \sum_{i=1}^d \pi_i g_i \ge 0 \text{ $\Pcal$-q.s.}\right\}.
\]
Conversely, every function of the above form is an admissible e-variable provided $\Phi$ satisfies the following constraint qualification:
\begin{equation}\label{C_max_evar_CQ}
\text{If $g, g' \in \cone(\Phi)$ and $g \le g'$, $\Pcal$-q.s., then $g = g'$, $\Pcal$-q.s.}
\end{equation}
\end{corollary}

\begin{proof}
The first part is immediate from Theorem~\ref{T_finitely_generated}. For the second part, fix an e-variable of the form $1+g$, where $g = \sum_{i=1}^d \pi_i g_i$ for some $\pi \in \Pi^\Phi$, and consider any e-variable $h$ that $\Pcal$-q.s.\ dominates $1+g$. We must show that $h = 1+g$, $\Pcal$-q.s. By Theorem~\ref{T_finitely_generated}, $h$ is $\Pcal$-q.s.\ dominated by an e-variable of the form $1 + g'$, where $g' = \sum_{i=1}^d \pi_i' g_i$ for some $\pi' \in \R^d_+$. We thus have $1+g \le h \le 1+g'$, $\Pcal$-q.s. The constraint qualification \eqref{C_max_evar_CQ} now yields $g = g'$, and hence $h = 1+g$, $\Pcal$-q.s. This shows that $h$ is admissible.
\end{proof}

\begin{remark}
In Corollary~\ref{C_evar_max_finite} we could further restrict $\pi$ to the set $K_0 = K \cap \Pi^\Phi$ with $K$ given in \eqref{eq_nonredundant_K}; see Lemma~\ref{L_not_redundant}. This is noteworthy because, as we show in Lemma~\ref{L_K0_compact_finite} below, the set $K_0$ is compact. This significantly simplifies the problem of finding \emph{optimal} e-variables, and is discussed in detail in Section~\ref{S:optimal}. It is worth mentioning that, is practice, $\Pi^\Phi$ itself is often already compact.
\end{remark}

As another corollary, we are able to describe all admissible nonasymptotically valid tests for $\Pcal$. Here a \emph{level-$\alpha$ test} for $\Pcal$, where $\alpha \in (0,1)$, is a $[0,1]$-valued measurable function $\varphi$ such that $\int_\Xcal \varphi d \mu \le \alpha$ for all $\mu \in \Pcal$. The test $\varphi$ is \emph{admissible} if whenever another such test $\varphi'$ satisfies $\varphi' \ge \varphi$, $\Pcal$-q.s., we actually have $\varphi' = \varphi$, $\Pcal$-q.s.

\begin{corollary} \label{C_admissible_test_finitely_constrained}
Fix $\alpha \in (0,1)$. Every admissible level-$\alpha$ test $\varphi$ for $\Pcal$ admits the representation $\varphi = 1 \wedge (\alpha + \sum_{i=1}^d \pi_i g_i)$, $\Pcal$-q.s., for some $\pi \in \Pi^\Phi$.
\end{corollary}

\begin{proof}
Note that $\varphi / \alpha$ is an e-variable for $\Pcal$, and is thus $\Pcal$-q.s.\ dominated by some e-variable of the form $h = 1 + \sum_{i=1}^d \pi'_i g_i$ with $\pi' \in \Pi^\Phi$. Define $\varphi' = 1 \wedge (\alpha + \sum_{i=1}^d \alpha \pi'_i g_i)$. Then $\varphi' \ge \varphi$, $\Pcal$-q.s., and, thanks to the e-variable property, $\varphi'$ is a level-$\alpha$ test for $\Pcal$. Since $\varphi$ was admissible, $\varphi' = \varphi$, $\Pcal$-q.s. Setting $\pi = \alpha \pi'$ we get the claimed representation.
\end{proof}

The representation in Corollary~\ref{C_evar_max_finite} is useful because the set $\Pi^\Phi$ can be described explicitly in various cases of interest. The following simple, yet interesting, example illustrates this; see \cite{agrawal2020optimal,agrawal2021regret,clerico2024optimalevaluetestingproperly,wang2023catoni,fan2025testing} for more details.

\begin{example} \label{ex:250302}
We take $\Xcal = \R$ and let $\Pcal$ consist of all zero mean distributions with standard deviation bounded by a positive number $\sigma$. This hypothesis is generated by the constraint set $\Phi = \{x, -x, x^2 - \sigma^2\}$. Thus $\Pi^\Phi$ consists of all $(\pi_1,\pi_2,\pi_3) \in \R^3_+$ such that $1+(\pi_1-\pi_2)x + \pi_3(x^2-\sigma^2) \ge 0$ for all $x \in \R$. (The only $\Pcal$-negligible set is the empty set, which is why the inequality must hold for \emph{all} $x$.) It is natural to re-parameterize in terms of $\alpha = \pi_1-\pi_2 \in \R$ and $\beta = \pi_3\sigma^2 \in \R_+$, constrained to satisfy $1 + \alpha x + \beta(x^2/\sigma^2 - 1) \ge 0$ for all $x \in \R$. Minimizing over $x$ and requiring that the minimum value be nonnegative, one arrives at the constraint $\sigma^2\alpha^2 + (2\beta - 1)^2 \le 1$ on $\alpha,\beta$. We conclude that every admissible e-variable is of the form
\begin{equation}\label{ex_mean_var_max_evars}
1 + \alpha x + \beta \left(\frac{x^2}{\sigma^2} - 1\right), \quad x \in \R,
\end{equation}
for some $(\alpha,\beta)$ inside the ellipse determined by $\sigma^2\alpha^2 + (2\beta - 1)^2 \le 1$. Note we do not have to impose $\beta \ge 0$ separately, since this is already implied by the ellipse constraint. Finally, $\Phi$ satisfies the constraint qualification \eqref{C_max_evar_CQ}. Indeed, if $\alpha x + \beta (x^2/\sigma^2-1) \ge \alpha' x + \beta'(x^2/\sigma^2-1)$ for all $x \in \R$, we first take $x = \pm\sigma$ to get $\alpha = \alpha'$, and then (say) $x = 0$ and $x = 2\sigma$ to get $\beta = \beta'$. Consequently, every function of the form \eqref{ex_mean_var_max_evars} is an admissible e-variable.
\end{example}

The fact that the constraint functions are not required to satisfy any kind of continuity or other regularity conditions beyond measurability is sometimes useful, for instance in the context of quantiles.

\begin{example} 
We continue to take $\Xcal = \R$. Fix $\alpha \in (0,1)$ and $q \in \R$, and let $\Pcal$ consist of all distributions $\mu$ whose $\alpha$-quantile is at most $q$, meaning that $\mu((-\infty, q]) \ge \alpha$. This hypothesis is generated by the single constraint function $\alpha - \bm1_{(-\infty,q]}(x)$. Since this function takes both positive and negative values, the constraint qualification \eqref{C_max_evar_CQ} holds. Thus the admissible e-variables are the functions $1 + \pi_1(\alpha - \bm1_{(-\infty,q]}(x))$ with $\pi_1 \in [0,(1-\alpha)^{-1}]$ to ensure nonnegativity.
\end{example}

The following example is common in the recent literature involving the mean of a bounded random variable.

\begin{example}\label{eg:bounded-mean}
Take $\Xcal = [0,1]$ and let $\Pcal$ consist of all distributions whose mean is at most a given constant $m \in (0,1)$. This hypothesis is generated by the constraint function $x - m$, and the constraint qualification \eqref{C_max_evar_CQ} holds. Thus, the admissible e-variables are the functions $1+ \pi_1 (x-m)$ with $\pi_1 \in [0,1/m]$ to ensure nonnegativity. If we further add the constraint function $m-x$, then the resulting (smaller) class $\Pcal$ consists of those distributions with mean equal to $m$. In this case, the admissible e-variables have the same form, but with $\pi_1 \in [-1/(1-m),1/m]$.
\end{example}

In particular, this recovers the class of e-variables used in~\cite{waudby2024estimating,larsson2024numeraire,orabona2023tight,clerico2024optimality}. A minor variant of Example~\ref{ex:250302} shows that without the boundedness assumption, there do not exist any nontrivial e-variables.

\begin{example} \label{ex_unbounded_mean}
Take $\Xcal = \R$ and let $\Pcal$ consist of all distributions whose mean exists and equals zero. This hypothesis is generated by the constraint functions $x$ and $-x$, and the constraint qualification \eqref{C_max_evar_CQ} holds. Thus, the admissible e-variables are the functions $1 + (\pi_1 - \pi_2)x = 1+ \alpha x$, where we reparameterize in terms of $\alpha = \pi_1-\pi_2 \in \R$ as in Example~\ref{ex:250302}. We must choose $\alpha$ so that $\alpha x$ is nonnegative for any $x \in \R$. This immediately implies $\alpha=0$, showing that the e-variable equal to one is the only admissible e-variable in this class (and all other e-variables must be less than or equal to one). The same argument shows that the hypothesis of a well-defined and nonpositive mean, generated by the single constraint function $x$, also only admits trivial e-variables.
\end{example}

We end this section with a brief discussion of reference measures, building on Remark~\ref{R_ref_measure}. In addition to the constraint set $\Phi = \{g_1,\ldots,g_d\}$, suppose a desired reference measure $\bar\mu$ is given. We are interested in the hypothesis
\[
\Pcal_{\bar\mu} = \{\mu \in \Pcal \colon \mu \ll \bar\mu\}.
\]
Note that $\Pcal_{\bar\mu}$ is generated by the augmented constraint set $\Phi_{\bar\mu} = \Phi \cup \{\bm1_A \colon \bar\mu(A) = 0\}$. 
Moreover, functions which are zero $\bar\mu$-a.e.\ are actually $\Pcal_{\bar\mu}$-negligible, including all conic combinations of the constraint functions $\bm1_A$ with $\bar\mu(A) = 0$. Therefore the set $\Ccal$ in \eqref{eq_Ccal_new}, with $\Phi_{\bar\mu}$ in place of $\Phi$, consists of all functions of the form $f = \sum_{i=1}^d \pi_i g_i - h$ with $\pi_1,\ldots,\pi_d \in \R_+$ and $h \ge 0$, $\Pcal_{\bar\mu}$-q.s. With this observation in hand we obtain the following result, whose proof is word for word the same as that of Theorem~\ref{T_finitely_generated} (including the proof of Lemma~\ref{L_not_redundant}) and Corollary~\ref{C_evar_max_finite}, replacing only $\Phi$ with $\Phi_{\bar\mu}$ and $\Pcal$ with $\Pcal_{\bar\mu}$.

\begin{theorem} \label{T_finite_constraint_set_reference_measure}
The statements of Theorem~\ref{T_finitely_generated} and Corollary~\ref{C_evar_max_finite} remain true with $\Phi_{\bar\mu}$ and $\Pcal_{\bar\mu}$ in place of $\Phi$ and $\Pcal$.
\end{theorem}

\section{One-sided sub-$\psi$ distributions} \label{S_sub_psi}

Fix a closed convex function $\psi \colon \R \to \R \cup \{\infty\}$ whose effective domain $\dom(\psi)$ is either $[0,\lambda_\text{max})$ for some $\lambda_\text{max} \in (0,\infty]$, or $[0,\lambda_\text{max}]$ for some $\lambda_\text{max} \in (0,\infty)$.\footnote{The effective domain is the set $\dom(\psi) = \{\lambda \in \R \colon \psi(\lambda) < \infty\}$ where $\psi$ is finite. That $\psi$ is closed means that its epigraph $\{(\lambda,y) \in \R \times \R \colon \psi(\lambda) \le y\}$ is closed. For our $\psi$, this just says that $\psi$ is continuous at $\lambda_\text{max}$ (if this is finite) and at $0$.} We assume that $\psi$ is nonnegative and that $\psi(0) = 0$. Key examples include cumulant generating functions of zero-mean distributions, modified to take the value infinity on the negative half-line. More generally, $\psi$ could be a \emph{CGF-like function} in the terminology of~\cite{howard2020time,howard2021time}.
Our goal is to describe the e-variables for the hypothesis consisting of all distributions of the following kind, whose usage stems back to the work of~\cite{cramer1994nouveau} (originally 1938) and \cite{chernoff1952measure}.

\begin{definition} \label{D_sub_psi}
A probability measure $\mu \in \Mcal_1(\R)$ is called \emph{(one-sided) sub-$\psi$} if its cumulant generating function is bounded above by $\psi$,
that is,
\[
\int_\R e^{\lambda x} \mu(dx) \le e^{\psi(\lambda)} \text{ for all } \lambda \in \dom(\psi). 
\]
\end{definition}

For example, when $\psi(\lambda)=\sigma^2 \lambda^2/2$ and $\lambda_{\max}=\infty$, the measure is called $\sigma$-sub-Gaussian in the sense that its (right) tail is lighter than that of a centered Gaussian with variance $\sigma^2$.
The sub-$\psi$ property in Definition~\ref{D_sub_psi} is ``one-sided'' in the sense that no condition is imposed for negative values of $\lambda$. Several of the proofs below make use of this property. To simplify terminology we will usually omit the qualifier ``one-sided''.

The convex conjugate of $\psi$ is the convex function $\psi^*$ given by
\[
\psi^*(x) = \sup_{\lambda \in \R} \{ \lambda x - \psi(\lambda) \}, \quad x \in \R.
\]
Since $\psi(0) = 0$, $\psi^*$ takes values in $[0,\infty]$. Moreover, because $\psi$ is nonnegative and $\psi(\lambda) = \infty$ for $\lambda < 0$, we have $\psi^*(x) = 0$ for $x \le 0$. Thus the effective domain $\dom(\psi^*)$ contains the negative half-line.

\begin{lemma} \label{L_sub_psi_support}
Every sub-$\psi$ distribution is concentrated on $\dom(\psi^*)$.
\end{lemma}

\begin{proof}
This follows from the well-known fact that any sub-$\psi$ distributed random variable $X$ satisfies the Chernoff tail bound $\P(X \ge x) \le e^{-\psi^*(x)}$ for all $x$. Indeed, if $\dom(\psi^*) = \R$ there is nothing to prove. Otherwise, let $\bar x < \infty$ denote the right endpoint of the interval $\dom(\psi^*)$. If $\bar x \notin \dom(\psi^*)$, then $\psi^*(\bar x) = \infty$ and hence $\P(X \ge \bar x) \le e^{-\psi^*(\bar x)} = 0$, so that $X$ is concentrated on $\dom(\psi^*)$. If $\bar x \in \dom(\psi^*)$, then $\P(X > \bar x) = \lim_{x \downarrow \bar x} \P(X \ge x) = 0$, showing that $X$ is concentrated on $\dom(\psi^*)$ in this case too. 
\end{proof}

Thanks to Lemma~\ref{L_sub_psi_support}, any sub-$\psi$ distribution can be regarded as a probability measure on $\dom(\psi^*)$. We thus take $\Xcal = \dom(\psi^*)$ with its Borel $\sigma$-algebra. Consider the infinitely many (even uncountably many) constraint functions
\[
g_\lambda(x) = e^{\lambda x - \psi(\lambda)} - 1, \quad x \in \dom(\psi^*),
\]
indexed by $\lambda \in \dom(\psi)$. The hypothesis generated by these functions is
\begin{equation}\label{eq_sub_psi_hypothesis_new}
\Pcal = \{ \mu \in \Mcal_1 \colon \mu \text{ is sub-$\psi$} \}.
\end{equation}
If $\lambda_\text{max}$ is not already in $\dom(\psi)$, we include the additional constraint function
\[
g_{\lambda_\text{max}}(x) = \lim_{\lambda \uparrow \lambda_\text{max}} g_\lambda(x), \quad x \in \dom(\psi^*).
\]
The limit exists and is finite because $\lambda \mapsto \lambda x - \psi(\lambda)$ is concave and $-1 \le g_\lambda(x) \le e^{\psi^*(x)} - 1$. Fatou's lemma implies that $\int_\Xcal g_{\lambda_\text{max}} d\mu \le 0$ for any sub-$\psi$ distribution $\mu$, so $g_{\lambda_\text{max}}$ is redundant in the sense that including it does not alter the generated hypothesis $\Pcal$ in \eqref{eq_sub_psi_hypothesis_new}. It does, however, play a role in the representation theorem below. Note that the constraint functions $g_\lambda$ are now indexed by the compact set
\[
\Lambda = [0,\lambda_\text{max}]
\]
(where we stress that $\lambda_\text{max}$ may be infinity), and the maps $\lambda \mapsto g_\lambda(x)$ are continuous on $\Lambda$ for every $x \in \dom(\psi^*)$. Our full constraint set is
\[
\Phi = \{g_\lambda\colon \lambda \in \Lambda\}.
\]

\begin{theorem} \label{T_sub_psi_rep}
\begin{enumerate}
\item \label{T_sub_psi_rep_1} A function $f \in \Lcal^\Phi$ satisfies $\int_\Xcal f d\mu \le 0$ for all $\mu \in \Pcal$ if and only if
\[
f(x) \le \int_\Lambda g_\lambda(x) \pi(d\lambda), \quad x \in \dom(\psi^*),
\]
for some $\pi \in \Mcal_+(\Lambda)$. 
\item \label{T_sub_psi_rep_2} The set $\Ecal$ of all e-variables for $\Pcal$ consists precisely of those $[0,\infty]$-valued measurable functions that are pointwise dominated on $\dom(\psi^*)$ by
\[
\int_\Lambda e^{\lambda x - \psi(\lambda)} \pi(d\lambda)
\]
for some $\pi \in \Mcal_1(\Lambda)$. (Note that $\pi$ is a probability measure here. Moreover, if $\lambda_\textnormal{max} \notin \dom(\psi)$, the integrand for this value of $\lambda$ is understood as $\lim_{\lambda \uparrow \lambda_\textnormal{max}} e^{\lambda x - \psi(\lambda)}$.)
\end{enumerate}
\end{theorem}

Just as in the finitely generated case (see Corollary~\ref{C_admissible_test_finitely_constrained}), we obtain a description of all admissible tests for $\Pcal$. We do not repeat the proof.

\begin{corollary} \label{C_admissible_test_sub_psi}
Fix $\alpha \in (0,1)$. Every admissible level-$\alpha$ test $\varphi$ for $\Pcal$ admits the representation $\varphi(x) = 1 \wedge  \int_\Lambda \alpha e^{\lambda x - \psi(\lambda)} \pi(d\lambda)$ for $x \in \dom(\psi^*)$ for some $\pi \in \Mcal_1(\Lambda)$.
\end{corollary}

We now give a brief roadmap of the proof of Theorem~\ref{T_sub_psi_rep}, introducing some notation along the way. The most involved part of the proof is to show part~\ref{T_sub_psi_rep_1}. Once this has been done, the proof of part~\ref{T_sub_psi_rep_2} is straightforward.

To deduce Theorem~\ref{T_sub_psi_rep}\ref{T_sub_psi_rep_1} from Theorem~\ref{T_abs_rep}\ref{T_abs_rep_1}, it suffices to show that the weak closure $\overline\Ccal$ of the set $\Ccal = \cone(\Phi) - \Lcal^\Phi_p$ in \eqref{eq_Ccal_new} is equal to $\Gcal - \Lcal^\Phi_+$, where we define the convex cone
\begin{align} \label{eq:240309}
\Gcal = \left\{ \int_\Lambda g_\lambda \pi(d\lambda) \colon \pi \in \Mcal_+(\Lambda) \right\}.
\end{align}
By considering finitely supported measures $\pi$, one sees that $\Gcal$ contains $\cone(\Phi)$. Moreover, it is shown in Lemma~\ref{L_sub_psi_empty_set} below that the only $\Pcal$-negligible subset of $\dom(\psi^*)$ is the empty set, and hence $\Lcal^\Phi_p = \Lcal^\Phi_+$. Thus $\Ccal \subset \Gcal - \Lcal^\Phi_+$. On the other hand, for any $g = \int_\Lambda g_\lambda \pi(d\lambda) \in \Gcal$ and $\mu \in \Pcal$ we have from Tonelli's theorem that
\[
\int_\Xcal g(x) \mu(dx) = \int_\Lambda \int_\Xcal g_\lambda(x) \mu(dx) \pi(d\lambda) \le 0.
\]
This shows, first, that $\Gcal$ is indeed a subset of $\Lcal^\Phi$. It also shows, via the forward implication of Theorem~\ref{T_abs_rep}\ref{T_abs_rep_1}, that $\Gcal - \Lcal^\Phi_+ \subset \overline\Ccal$. In summary, we have
\[
\Ccal \subset \Gcal - \Lcal^\Phi_+ \subset \overline\Ccal.
\]
Therefore, to show that $\Gcal - \Lcal^\Phi_+ = \overline\Ccal$ it is enough to show that
\begin{equation} \label{eq_sub_psi_small_closure}
\text{$\Gcal - \Lcal^\Phi_+$ is $\sigma(\Lcal^\Phi, \Mcal^\Phi)$-closed.}
\end{equation}
This is the heart of the matter, and the proof relies on a closedness criterion for convex subsets of Banach space duals known as the Krein--\v{S}mulian theorem (see Appendix~\ref{S_krein_smulian}).

Unfortunately the Krein--\v{S}mulian theorem cannot be applied directly, because $\Lcal^\Phi$ is not the dual of a Banach space. Instead, we first endow $\Lcal^\Phi$ with a slightly weaker topology than $\sigma(\Lcal^\Phi, \Mcal^\Phi)$, which allows us to embed it into a larger space that \emph{is} the dual of a Banach space. Checking closedness in the weaker topology can now be done using the Krein--\v{S}mulian theorem. As we show below, this amounts to checking that for each $r \in \R_+$, the subset $\Gcal_r = \{g \in \Gcal \colon g \ge -r\} \subset \Gcal$ of elements uniformly bounded below by $r$ is compact. This turns out to be fairly straightforward, because $\Gcal_r$ is a continuous image of the compact set $\{\pi \in \Mcal_+(\Lambda) \colon \pi(\Lambda) \le r\}$ equipped with the usual weak topology coming from duality with the continuous functions on $\Lambda$.

The details of this argument depend on several preliminary results.

\begin{lemma} \label{L_sub_psi_two_points_measure}
Let $x_0 \in \dom(\psi^*)$. There exists $y_0 \ge 0$ such that for any $p \in [0,\frac12 e^{-\psi^*(x_0)}]$ and $y \ge y_0$, the probability measure $\nu = p \delta_{x_0} + (1-p) \delta_{-x_0-y}$ is sub-$\psi$.
\end{lemma}

\begin{proof}
Define
\[
f(\lambda,p,y) = \int_\Xcal e^{\lambda x - \psi(\lambda)} \nu(dx) = p e^{\lambda x_0 - \psi(\lambda)} + (1-p)e^{-\lambda(x_0 + y) - \psi(\lambda)}.
\]
If $x_0 \le 0$, the right-hand side is bounded by one for any $p \in [0,1]$,  $y \ge -x_0$, and $\lambda \in \R$, recalling that $\psi$ is nonnegative and that $\psi(\lambda) = \infty$ for $\lambda < 0$. Thus $\nu$ is sub-$\psi$, and we may  take $y_0 = -x_0$.

Consider now the case $x_0 > 0$. We have $f(\lambda, \frac12, 1) = \frac12 (e^{\lambda x_0} + e^{-\lambda(x_0 + 1)}) e^{-\psi(\lambda)}$, which is strictly decreasing with respect to $\lambda$ in a right neighborhood of zero. This is because $\psi$ is nondecreasing on $[0,\infty)$, being a nonnegative convex function, and $\frac{d}{d\lambda}|_{\lambda=0}(e^{\lambda x_0} + e^{-\lambda(x_0 + 1)}) = -1$. Moreover, $f(0,\frac12,1) = 1$,
so there exists some $\lambda_0 > 0$ such that $f(\lambda,\frac12,1) \le 1$ for all $\lambda \in [0,\lambda_0]$. For $p \in [0,\frac12]$ and $y \ge 1$, $f(\lambda,p,y)$ is nondecreasing in $p$ and nonincreasing in $y$, so it follows that $f(\lambda,p,y) \le 1$ for all such $p$ and $y$, provided $\lambda \in [0,\lambda_0]$. 

For $\lambda > \lambda_0$, $p \in [0,\frac12 e^{-\psi^*(x_0)}]$, and $y \ge 1$, we have the bound $f(\lambda,p,y) \le \frac12 + e^{-\lambda_0(x_0 + y) - \psi(\lambda_0)}$. This uses that $x_0 > 0$ and that $\psi$ is nondecreasing on $[0,\infty)$. The right-hand side is bounded by one for all $y \ge y_0$, where we may take $y_0 = \max\{1,\lambda_0^{-1}(\log(2) - \psi(\lambda_0)) - x_0\}$. In summary, we have the sub-$\psi$ inequality $f(\lambda,p,y) \le 1$ for all $\lambda \in \R$, provided $p \in [0,\frac12 e^{-\psi^*(x_0)}]$ and $y \ge y_0$.
\end{proof}

\begin{lemma} \label{L_sub_psi_empty_set}
The only subset of $\dom(\psi^*)$ that is $\Pcal$-negligible is the empty set.
\end{lemma}

\begin{proof}
Let $A$ be any nonempty measurable subset of $\dom(\psi^*)$ and pick $x_0 \in A$. Lemma~\ref{L_sub_psi_two_points_measure} yields a sub-$\psi$ distribution that charges $x_0$. Thus $A$ is not $\Pcal$-negligible.
\end{proof}

\begin{lemma} \label{L_sub_psi_weight_function}
Every $f \in \Lcal^\Phi$ satisfies $\sup_{x \in \dom(\psi^*)} |f(x)| e^{-\psi^*(x)} < \infty$.
\end{lemma}

\begin{proof}
We prove the contrapositive. Let $f$ be measurable with $\sup_{x \in \dom(\psi^*)} |f(x)| e^{-\psi^*(x)} = \infty$. Then for each $n \in \N$, there exists $x_n \in \dom(\psi^*)$ such that $|f(x_n)| \ge 2^n e^{\psi^*(x_n)}$. Furthermore, Lemma~\ref{L_sub_psi_two_points_measure} yields $y_n \ge 0$ such that the probability measure $\nu_n = p_n \delta_{x_n} + (1-p_n) \delta_{-x_n-y_n}$ with $p_n = \frac12 e^{-\psi^*(x_n)}$ is sub-$\psi$. Then so is the mixture $\mu = \sum_{n \in \N} 2^{-n} \nu_n$. On the other hand,
\[
\int_\Xcal |f(x)| \mu(dx) = \sum_{n \in \N} 2^{-n} \int_\Xcal |f(x)| \nu_n(dx) \ge \sum_{n \in \N} 2^{-n} p_n |f(x_n)| \ge \sum_{n \in \N} p_n e^{\psi^*(x_n)} = \infty.
\]
This shows that $f$ does not belong to $\Lcal^\Phi$, and completes the proof.
\end{proof}

\begin{proof}[Proof of Theorem~\ref{T_sub_psi_rep}\ref{T_sub_psi_rep_1}]
We will make use of the space
\[
E = \{\mu \in \Mcal \colon \int_\Xcal e^{\psi^*} d|\mu| < \infty\},
\]
which is a Banach space with the weighted total variation norm $\|\mu\| = \int_\Xcal e^{\psi^*} d|\mu|$. The Banach space dual $E'$ is itself a Banach space with the dual norm $\|\varphi\|' = \sup\{\varphi(\mu)\colon \mu \in E, \|\mu\| \le 1\}$, and admits the weak$^*$ topology $\sigma(E',E)$; see Appendix~\ref{S_krein_smulian}. The positive cones of $E$ and $E'$ are $E_+ = \{\mu \in E\colon \mu \ge 0\}$ and $E'_+ = \{\varphi \in E' \colon \varphi(\mu) \ge 0 \text{ for all } \mu \in E_+\}$.

For any $f \in \Lcal^\Phi$ and $\mu \in E$ we have
\begin{equation} \label{eq_sub_psi_bdd_lin}
\int_\Xcal |f| d|\mu| = \int_\Xcal |f| e^{-\psi^*} e^{\psi^*} d|\mu| \le c_f \| \mu \|,
\end{equation}
where $c_f = \sup |f| e^{-\psi^*}$ is finite thanks to Lemma~\ref{L_sub_psi_weight_function}. From \eqref{eq_sub_psi_bdd_lin} it follows that every $f \in \Lcal^\Phi$ defines a bounded linear functional $\varphi_f(\mu) = \int_\Xcal f d\mu$ on $E$, and we may thus regard $\Lcal^\Phi$ as a subspace of $E'$. In particular, $E'$ contains $\Gcal$, which is defined in \eqref{eq:240309} and is a subset of $\Lcal^\Phi$. We will show below that
\begin{equation} \label{eq_sub_psi_big_closure}
\Gcal - E'_+ \text{ is $\sigma(E',E)$-closed.}
\end{equation}
Once this has been done, the proof of Theorem~\ref{T_sub_psi_rep}\ref{T_sub_psi_rep_1} is completed as follows. Observe that $\Lcal^\Phi \cap (\Gcal - E'_+) = \Gcal - \Lcal^\Phi \cap E'_+ = \Gcal - \Lcal^\Phi_+$, where the equality $\Lcal^\Phi \cap E'_+ = \Lcal^\Phi_+$ holds because a function $f \in \Lcal^\Phi$ is nonnegative if and only if $\int_\Xcal f d\mu \ge 0$ for all $\mu \in E_+$. Consequently, \eqref{eq_sub_psi_big_closure} implies that $\Gcal - \Lcal^\Phi_+$ is closed in $\sigma(\Lcal^\Phi, E)$, which is the trace of $\sigma(E',E)$ on $\Lcal^\Phi$. Now, thanks to \eqref{eq_sub_psi_bdd_lin}, $E$ is a subset of $\Mcal^\Phi$. Thus the topology $\sigma(\Lcal^\Phi,E)$ is weaker than $\sigma(\Lcal^\Phi, \Mcal^\Phi)$, and we conclude that $\Gcal - \Lcal^\Phi_+$ is closed in the latter topology as well. This establishes \eqref{eq_sub_psi_small_closure} and proves the first part of the theorem.

We are left with proving \eqref{eq_sub_psi_big_closure}. Thanks to the Krein--\v{S}mulian theorem (see Theorem~\ref{T_Krein_Smulian}), we only need to show that
\[
\text{$(\Gcal - E'_+) \cap B'_r$ is $\sigma(E', E)$-closed for every $r \in (0,\infty)$,}
\]
where $B'_r = \{\varphi \in E' \colon \|\varphi\|' \le r\}$ is the centered closed dual ball of radius $r$. Fix any $r \in (0,\infty)$. Let $\varphi = g - \eta \in (\Gcal - E'_+) \cap B'_r$ be arbitrary, and observe that for any $\mu \in E_+$ with $\|\mu\| \le 1$ we have
\[
-r \le -\|\varphi\|' \le \varphi(\mu) = \int_\Xcal g d\mu - \eta(\mu) \le \int_\Xcal g d\mu.
\]
By taking $\mu = \delta_x$ we find that $g \ge -r$ pointwise. We conclude from this that
\[
(\Gcal - E'_+) \cap B'_r = (\Gcal_r - E'_+) \cap B'_r,
\]
where $\Gcal_r = \{g \in \Gcal \colon g \ge -r\}$. We will argue that $\Gcal_r$ is $\sigma(E',E)$-compact. This will conclude the proof because $E'_+$ and $B'_r$ are both $\sigma(E',E)$-closed, the sum of a compact set and a closed set is closed, and the intersection of two closed sets is closed.

To show that $\Gcal_r$ is $\sigma(E',E)$-compact, we define
\[
K_r = \{\pi \in \Mcal_+(\Lambda) \colon \pi(\Lambda) \le r\},
\]
where we recall that $\Lambda = [0,\lambda_\text{max}]$ is compact. Then $K_r$ is a compact subset of $\Mcal_+(\Lambda)$, equipped with the usual weak topology induced by duality with the set of real-valued continuous functions on $\Lambda$. Next, define the map
\[
T \colon \pi \mapsto T(\pi) = \int_\Lambda g_\lambda \pi(d\lambda)
\]
from $K_r$ to $\Lcal^\Phi \subset E'$. Here we identify $T(\pi)$ with the linear functional $\varphi_{T(\pi)}(\mu) = \int_\Xcal T(\pi) d\mu$ on $E$. We claim that
\begin{equation} \label{eq_sub_psi_Gr_TKr}
\Gcal_r = T(K_r).
\end{equation}
The inclusion `$\supset$' is clear since $g_\lambda \ge -1$ for all $\lambda$. For the inclusion `$\subset$', consider any $g \in \Gcal_r$, that is, $g = \int_\Lambda g_\lambda \pi(d\lambda) \ge -r$ for some $\pi \in \Mcal_+(\Lambda)$. Now, for all $x < 0$ and $\lambda \in (0,\lambda_\text{max}]$, we have $-1 \le g_\lambda(x) \le 0$ and $\lim_{x\to-\infty} g_\lambda(x) = -1$, while $g_0(x) = 0$ for all $x$. The dominated convergence theorem then yields $-r \le \lim_{x \to -\infty} g(x) = -\pi((0,\lambda_\text{max}])$. Thus the measure $\pi' = \pi(\fdot \cap (0,\lambda_\text{max}])$ belongs to $K_r$, and we have $g = \int_\Lambda g_\lambda \pi'(d\lambda)$. This completes the proof of \eqref{eq_sub_psi_Gr_TKr}.

Next, we claim that $T$ is continuous when $E'$ is equipped with $\sigma(E',E)$. This is the initial topology generated by the maps $\varphi \mapsto \varphi(\mu)$, $\mu \in E$, so to show continuity it suffices to show that the composition
\begin{equation} \label{eq_sub_psi_composition_continuous_new}
\pi \mapsto \varphi_{T(\pi)}(\mu) = \int_\Xcal T(\pi) d\mu = \int_{\Lambda} \int_\Xcal g_\lambda(x) \mu(dx) \pi(d\lambda)
\end{equation}
from $\Mcal_+(\Lambda)$ to $\R$ is continuous for every $\mu \in E$. (We used Fubini's theorem to interchange the integrals on the right-hand side of \eqref{eq_sub_psi_composition_continuous_new}.)
The map $\lambda \mapsto \int_\Xcal g_\lambda(x) \mu(dx)$ is continuous on the compact set $\Lambda$. This follows from the dominated convergence theorem because $g_\lambda$ is continuous in $\lambda$ and dominated in absolute value by $e^{\psi^*}$ which is $\mu$-integrable by definition of $E$. Thus by definition of the weak topology on $\Mcal_+(\Lambda)$, the map in \eqref{eq_sub_psi_composition_continuous_new} is continuous. We conclude that $T$ is continuous, and hence that $\Gcal_r = T(K_r)$ is $\sigma(E',E)$-compact. This completes the proof of Theorem~\ref{T_sub_psi_rep}\ref{T_sub_psi_rep_1}.
\end{proof}

\begin{proof}[Proof of Theorem~\ref{T_sub_psi_rep}\ref{T_sub_psi_rep_2}]
Consider a $[0,\infty]$-valued measurable function $h$ that is pointwise dominated on $\dom(\psi^*)$ by $\int_\Lambda e^{\lambda x - \psi(\lambda)}\pi(d\lambda)$ for some $\pi \in \Mcal_1(\Lambda)$. Tonelli's theorem and the definition of $\Pcal$ then yields $\int_\Xcal h d\mu \le 1$ for all $\mu \in \Pcal$, showing that $h$ is an e-variable.

Conversely, let $h$ be an e-variable and set $f = h-1$. Then by part~\ref{T_sub_psi_rep_1} of the theorem, there is some $\pi' \in \Mcal_+(\Lambda)$ such that $f(x) \le \int_\Lambda g_\lambda(x) \pi'(d\lambda)$ for all $x \in \dom(\psi^*)$. Since $f \ge -1$, the argument after \eqref{eq_sub_psi_Gr_TKr} with $r=1$ yields $\pi'((0,\lambda_\text{max}]) \le 1$. Thus the measure $\pi = \pi'(\fdot \cap (0,\lambda_\text{max}]) + (1-\pi'((0,\lambda_\text{max}]))\delta_0$ belongs to $\Mcal_1(\Lambda)$. Since $g_0(x) = 0$ for all $x$, we have $\int_\Lambda g_\lambda(x) \pi'(d\lambda) = \int_\Lambda g_\lambda(x) \pi(d\lambda)$, and thus $h(x) = 1+f(x) \le\int_\Lambda (1 + g_\lambda(x))\pi(d\lambda) = \int_\Lambda e^{\lambda x - \psi(\lambda)} \pi(d\lambda)$,
for all $x \in \dom(\psi^*)$.
\end{proof}

\section{Distributions invariant under a group of symmetries}
\label{S_group_symmetry}

Let $\Sigma$ be a compact topological group acting (from the left) on the measurable space $\Xcal$. This means that every group element $\sigma \in \Sigma$ induces a map $x \mapsto \sigma x$ from $\Xcal$ to itself, the identity element of $\Sigma$ induces the identity map, and one has $(\sigma_1 \sigma_2) x = \sigma_1 (\sigma_2 x)$ for all $\sigma_1, \sigma_2 \in \Sigma$ and $x \in \Xcal$. We assume that the group action is measurable, meaning that the map $(\sigma,x) \mapsto \sigma x$ is jointly measurable, where $\Sigma$ is equipped with its Borel $\sigma$-algebra. Since $\Sigma$ is compact, it admits a unique left Haar probability measure $\pi$. Here are two examples of such group actions.

\begin{example}
\begin{enumerate}
\item The symmetric group on $n$ elements $\Sigma(n)$ acts on vectors in $\R^n$ by permuting the components. Its Haar probability measure is the normalized counting measure on $\Sigma(n)$.

\item The special orthogonal group $SO(n)$ acts on $\R^n$ by rotations. Its Haar probability measure is the uniform distribution on $SO(n)$.
\end{enumerate}
\end{example}

We use the left Haar probability measure $\pi$ to symmetrize measures and to average functions. First, for any measurable function $f$ bounded below, we define its \emph{orbit average} function $f_\pi$ by
\[
f_\pi(x) = \int_\Sigma (\sigma^*f)(x) \pi(d\sigma),
\]
where $(\sigma^* f)(x) = f(\sigma x)$ is the \emph{pullback} of $f$ under the map $x \mapsto \sigma x$.
Thus $f_\pi(x)$ is indeed the average of $f$ over the orbit $O_x = \{\sigma x \colon \sigma \in \Sigma\}$ of $x$. Next, there is a dual operation on measures (we focus on probability measures for simplicity). For any $\mu \in \Mcal_1$ we define its \emph{symmetrization} $\mu_\pi \in \Mcal_1$ by
\[
\mu_\pi(A) = \int_\Sigma (\sigma_* \mu)(A) \pi(d\sigma),
\]
where $(\sigma_* \mu)(A) = \mu(\sigma^{-1}A)$ is the \emph{pushforward} of $\mu$ under the map $x \mapsto \sigma x$. Here $\sigma^{-1} A = \{\sigma^{-1} x \colon x \in A\}$. The fact that $\mu_\pi$ has unit mass is seen by taking $A = \Xcal$ and using that $\mu$ and $\pi$ both have unit mass. The following lemma records some basic properties of the orbit averaging and symmetrization operations.

\begin{lemma} \label{L_symmetrizations}
Let $f$ be a measurable function bounded below and let $\mu \in \Mcal_1$.
\begin{enumerate}
\item\label{L_symmetrizations_1} The symmetrization $\mu_\pi$ is $\Sigma$-invariant in the sense that $\sigma_* \mu_\pi = \mu_\pi$ for all $\sigma \in \Sigma$.
\item\label{L_symmetrizations_2} One has the adjoint identity
\begin{equation}\label{eq_symmetrizaion_adjoints}
\int_\Xcal f_\pi d\mu = \int_\Xcal f d\mu_\pi.
\end{equation}
\end{enumerate}
\end{lemma}

\begin{proof}
\ref{L_symmetrizations_1}: For any $\sigma \in \Sigma$ and any measurable set $A \subset \Xcal$, one has $\int_\Sigma \mu((\sigma \rho)^{-1} A) \pi(d\rho) = \int_\Sigma \mu(\rho^{-1} A) \pi(d\rho)$ thanks to the left-invariance of $\pi$. The left-hand side equals $(\sigma_* \mu_\pi)(A)$ and the right-hand side equals $\mu_\pi(A)$, showing that the two are equal.

\ref{L_symmetrizations_2}: Linearity and the definition of pushforward yield $\int_\Xcal f d\mu_\pi = \int_\Sigma \int_\Xcal f(\sigma x) \mu(dx) \pi(d\sigma)$
for every simple function $f$, and then for every bounded measurable $f$ by the monotone class theorem. On the other hand, Fubini's theorem yields $\int_\Xcal f_\pi d\mu = \int_\Sigma \int_\Xcal f(\sigma x) \mu(dx) \pi(d\sigma)$ for bounded measurable $f$. This shows \eqref{eq_symmetrizaion_adjoints} for all such $f$. For $f$ unbounded above, just apply \eqref{eq_symmetrizaion_adjoints} with $f \wedge n$ in place of $f$, send $n$ to infinity, and use monotone convergence.
\end{proof}

\begin{remark}
Given the $\Sigma$-invariance property \ref{L_symmetrizations_1} of $\mu_\pi$, it is perhaps surprising that the analogous property does \emph{not} hold for $f_\pi$ in the sense that $\sigma^* f_\pi$ and $f_\pi$ are not equal in general. 
They are however equal if $\pi$ is a \emph{right} Haar measure, since then $f_\pi(\sigma x) = \int_\Sigma f_\pi(\rho \sigma x) \pi(d\rho) = \int_\Sigma f_\pi(\rho x) \pi(d\rho) = f_\pi(x)$, using the right-invariance of $\pi$ in the second step. If the group $\Sigma$ is unimodular, for example, if it is a discrete group, then $\pi$ is both a left and right Haar measure, and thus $\sigma^* f_\pi = f_\pi$.
\end{remark}

We are interested in describing the set of e-variables for the hypothesis consisting of all $\Sigma$-invariant distributions,
\[
\Pcal = \{\mu \in \Mcal_1 \colon \mu = \sigma_* \mu \text{ for all } \sigma \in \Sigma\}.
\]
Such classes, or infinite-sample versions of them, have been studied in many recent works. For example, testing exchangeability~\citep{vovk2021testing,ramdas2022testing,saha2024testing}, two-sample and independence testing~\citep{shekhar2023nonparametric,podkopaev2023sequential,podkopaev2023sequentialb},  but there are also papers that study this class in an abstract and general manner~\citep{koning2023post,pandeva2024deep}  like we do above. We note the subtle fact that our setting is different from the case where $\mu \neq \sigma_*\mu$, but $\sigma_*\mu \in \Pcal$ whenever $\mu \in \Pcal$, for which the term `group invariance' is also used~\citep{perez2024statistics}.


The following lemma is the key to characterizing the set of e-variables for $\Pcal$.

\begin{lemma} \label{L_symmetric_char}
A distribution $\mu \in \Mcal_1$ belongs to $\Pcal$ if and only if $\int_\Xcal f d\mu = \int_\Xcal f_\pi d\mu$ for all measurable functions $f$ bounded below.
\end{lemma}

\begin{proof}
Fix $\mu \in \Mcal_1$. We have the following chain of equivalences:
\begin{align*}
\mu \in \Pcal
\quad \Leftrightarrow \quad 
\mu = \mu_\pi
\quad &\Leftrightarrow \quad 
\int_\Xcal f d\mu = \int_\Xcal f d\mu_\pi \text{ for all measurable $f$ bounded below},
\end{align*}
where the first equivalence follows from the definition of $\Pcal$ and of $\mu_\pi$ for the forward implication, and Lemma~\ref{L_symmetrizations}\ref{L_symmetrizations_1} for the reverse implication. Thanks to \eqref{eq_symmetrizaion_adjoints}, the third statement in the above display is in turn equivalent to $\int_\Xcal f d\mu = \int_\Xcal f_\pi d\mu$ for all measurable functions $f$ bounded below. This completes the proof.
\end{proof}

\begin{theorem} \label{T_symmetric_rep}
The set $\Ecal$ of all e-variables for $\Pcal$ consists precisely of those $[0,\infty]$-valued measurable functions that are pointwise dominated by an e-variable of the form
\[
1 + f - f_\pi
\]
for some $[-1,\infty]$-valued measurable $f$ such that $f_\pi \le 0$. Moreover, every function of this form is an \emph{exact} e-variable, meaning that the e-variable property holds with equality for all $\mu \in \Pcal$.
\end{theorem}

\begin{proof}
Let $f$ be a $[-1,\infty]$-valued measurable function such that $f_\pi \le 0$, and note that, in addition, $f_\pi \ge -1$. This ensures that $h = 1 + f - f_\pi$ is a well-defined $[0,\infty]$-valued function, and that we may compute $\int_\Xcal h d\mu = 1 + \int_\Xcal f d\mu - \int_\Xcal f_\pi d\mu = 1$ for every $\mu \in \Pcal$, using Lemma~\ref{L_symmetric_char} in the last step. This shows that every function of this form is an exact e-variable.

Conversely, let $h$ be an e-variable and set $f = h - 1$. Fix any $x_0 \in \Xcal$ and consider the symmetrization $\mu = (\delta_{x_0})_\pi$ of the Dirac mass at $x_0$, which belongs to $\Pcal$ due to Lemma~\ref{L_symmetrizations}\ref{L_symmetrizations_1}. Thanks to \eqref{eq_symmetrizaion_adjoints} and the e-variable property of $h$ we have $f_\pi(x_0) = \int_\Xcal f_\pi(x) \delta_{x_0}(dx) = \int_\Xcal f(x) (\delta_{x_0})_\pi(dx) = \int_\Xcal f d\mu \le 0$. Thus $f_\pi \le 0$, and we have $h \le 1 + f - f_\pi$ pointwise.
\end{proof}

So far we have not made use of the abstract characterization of e-variables, Theorem~\ref{T_abs_rep}. Indeed, we were able to describe $\Ecal$ completely without it. We may however use the abstract theorem in a different way: our next result provides a general method of identifying constraint sets $\Phi$ that generate $\Pcal$, and the abstract theorem then ensures that any e-variable can be approximated in the weak sense using conic combinations of the constraint functions.

A set $\Fcal$ of bounded measurable functions is a \emph{separating set} for $\Mcal$ if, for any $\mu \in \Mcal$, one has $\mu = 0$ if and only if $\int_\Xcal f d\mu = 0$ for all $f \in \Fcal$. Such a set separates any two distinct measures $\mu_1,\mu_2$ in the sense that there is some $f \in \Fcal$ such that $\int_\Xcal f d\mu_1 \ne \int_\Xcal f d\mu_2$. Next, a \emph{generating set} for $\Sigma$ is a subset $\Sigma_0$ such that any $\sigma \in \Sigma$ can be expressed as $\sigma = \sigma_1 \sigma_2 \cdots \sigma_n$ for some $n \in \N$ and $\sigma_1,\ldots,\sigma_n \in \Sigma_0$.

\begin{theorem} \label{T_group_inv_constr_set}
Let $\Fcal$ be a separating set for $\Mcal$, and $\Sigma_0$ a generating set for $\Sigma$. A distribution $\mu$ belongs to $\Pcal$ if and only if $\int_\Xcal (f(\sigma x) - f(x)) \mu(dx) = 0$ for all $\sigma \in \Sigma_0$ and $f \in \Fcal$. In other words, $\Pcal$ is generated by the constraint set
\[
\Phi = \{\sigma^*f - f \colon \sigma \in \Sigma_0, \ f \in \Fcal \cup (-\Fcal) \}.
\]
\end{theorem}

\begin{proof}
The forward implication follows because for any $\mu \in \Pcal$ and any bounded measurable function $f$, $\int_\Xcal f(\sigma x) \mu(dx) = \int_\Xcal f(x) (\sigma_*\mu)(dx) = \int_\Xcal f(x) \mu(dx)$. We thus focus on the reverse implication and assume that
\begin{equation}\label{T_group_inv_constr_set_1}
\int_\Xcal (f(\sigma x) - f(x)) \mu(dx) = 0
\end{equation}
for all $\sigma \in \Sigma_0$ and $f \in \Fcal$, where $\mu \in \Mcal_1$ is fixed. We let $\Lcal_b$ denote the space of all bounded measurable functions on $\Xcal$.

Because $\Fcal$ is separating, its span is $\sigma(\Lcal_b,\Mcal)$-dense in $\Lcal_b$. Indeed, if the span were not dense, the Hahn--Banach theorem would yield a nonzero measure vanishing on the span, contradicting that $\Fcal$ is separating. Furthermore, the map $f \mapsto \int_\Xcal f(\sigma x) \mu(dx) = \int_\Xcal f d(\sigma_* \mu)$ is $\sigma(\Lcal_b,\Mcal)$-continuous by definition of the topology. Combining these two facts, we deduce that \eqref{T_group_inv_constr_set_1} holds for all $\sigma \in \Sigma_0$ and all $f \in \Lcal_b$.

Next, fix any $\sigma \in \Sigma$ and $f \in \Lcal_b$, write $\sigma = \sigma_1 \sigma_2 \cdots \sigma_n$ for some $\sigma_1,\ldots,\sigma_n \in \Sigma_0$, and set $f_i(x) = f(\sigma_1\cdots\sigma_i x)$ for $i=1,\ldots,n-1$ as well as $f_0(x) = f(x)$. We then have $f_i(x) = f_{i-1}(\sigma_i x)$ and $f_{n-1}(\sigma_n x) = f(\sigma x)$, and hence also the telescoping sum
\[
f(\sigma x) - f(x) = \sum_{i=0}^{n-1} (f_i(\sigma_i x) - f_i(x)).
\]
Since each $f_i$ belongs to $\Lcal_b$, we may use that \eqref{T_group_inv_constr_set_1} holds for functions in $\Lcal_b$ and group elements in $\Sigma_0$ to obtain
\[
\int_\Xcal (f(\sigma x) - f(x)) \mu(dx) = \sum_{i=0}^{n-1} \int_\Xcal (f_i(\sigma_i x) - f_i(x)) \mu(dx) = 0.
\]
This shows that \eqref{T_group_inv_constr_set_1} actually holds for all $\sigma \in \Sigma$ and all $f \in \Lcal_b$. By integrating over $\Sigma$ and using Fubini's theorem, we obtain $\int_\Xcal f_\pi d\mu = \int_\Xcal f d\mu$ for all $f \in \Lcal_b$, and then by monotone convergence for all measurable $f$ bounded below. Lemma~\ref{L_symmetric_char} now yields $\mu \in \Pcal$.
\end{proof}

\begin{example}
A finite sequence of random variables $X_1,\ldots,X_n$ is exchangeable if and only if its joint distribution is invariant under the symmetric group $\Sigma(n)$. The cardinality of $\Sigma(n)$ is $n!$, and the (left and right) Haar probability measure $\pi$ is simply the uniform distribution assigning mass $1/n!$ to each permutation.
The symmetrization of any probability measure $\mu$ on $\R^n$ is then
\[
\mu_\pi(A) = \frac{1}{n!} \sum_{\sigma \in \Sigma(n)} \mu(\sigma^{-1}A) = \frac{1}{n!} \sum_{\sigma \in \Sigma(n)} \mu(\sigma A).
\]
Moreover, the e-variables in Theorem~\ref{T_symmetric_rep} take the form
\[
1 + f(x) - \frac{1}{n!} \sum_{\sigma \in \Sigma(n)} f(\sigma x),
\]
for some $[-1,\infty)$-valued measurable $f$ such that its symmetrization $\frac{1}{n!} \sum_{\sigma \in \Sigma(n)} f(\sigma x)$ is nonpositive. In particular, all e-variables for this hypothesis are finite, because otherwise the symmetrization could not possibly be nonpositive (and the same would hold for any finite group).
Alternatively, one can actually write any e-variable as
\[
\frac{g(x)}{\frac{1}{n!} \sum_{\sigma \in \Sigma(n)} g(\sigma x)},
\]
using any $[0,\infty)$-valued measurable function $g$ and interpreting $0/0=1$.
One can move between these two representations easily. In one direction, given $f$, define $g(x) = 1 + f(x) - \frac{1}{n!} \sum_{\sigma \in \Sigma(n)} f(\sigma x)$ and note that its symmetrization equals one to conclude that the above two expressions are equal. In the other direction, given $g$, define $f(x) = (g(x) / \frac{1}{n!} \sum_{\sigma \in \Sigma(n)} g(\sigma x)) - 1$ and note that its symmetrization equals zero  to reach the same conclusion.

Next, while $\Sigma(n)$ contains $n!$ elements, it admits small generating sets. For instance, in the context of Theorem~\ref{T_group_inv_constr_set}, one may take $\Sigma_0$ to consist of only two permutations: the transposition $\sigma_{12}: (1,2,3,\ldots,n) \mapsto (2,1,3,\ldots,n)$ and the cyclic permutation $\sigma_\circ: (1,2,\ldots,n) \mapsto (n,1,2,\ldots,n-1)$. One could take $\mathcal F$ to be, for example, the set of all bounded continuous functions. Theorem~\ref{T_group_inv_constr_set} then yields that the set $\Pcal$ of exchangeable distributions on $n$ observations is generated by 
\[
\Phi = \{\sigma_{12}^*f - f, \sigma^*_\circ f - f \colon  \ f \in \Fcal \}.
\]

Finally, we note that we could have also derived these e-variables by conditioning on the unordered set of observed data, which we represent here using the order statistics, $X_{()} = (X_{(1)},\dots,X_{(n)})$. Effectively, conditioning on $X_{()}$ collapses the data distribution on $\mathbb R^n$  to a distribution on the orbit $\{\sigma X_{()}: \sigma \in \Sigma\}$ having $n!$ elements. 
Under the null $\Pcal$, this conditional distribution collapses to a point null hypothesis, which is a uniform distribution over the orbit.
The above e-variables are \emph{conditional e-variables} in the sense that they satisfy
\[
\E_\mu\left[\left.\frac{g(x)}{\frac{1}{n!} \sum_{\sigma \in \Sigma(n)} g(\sigma x)} \, \right| \, X_{()} \right] \le 1, \quad \mu \in \Pcal.
\]
In fact, one can straightforwardly interpret the above e-variable as a likelihood ratio, with the denominator being the aforementioned uniform, to read off the implied alternative in the numerator, against which this e-variable is in fact log-optimal. See also~\cite{perez2024statistics} for the study of log-optimal e-variables in this context. 
\end{example}

\section{Admissible e-variables and complete classes} \label{S_admissibility}

Consider a hypothesis $\Pcal$ and its set of e-variables $\Ecal$. Recall that an e-variable $h \in \Ecal$ is \emph{admissible} if any other $h' \in \Ecal$ with $h' \ge h$, $\Pcal$-q.s., actually satisfies $h' = h$, $\Pcal$-q.s. The results in this paper, in particular Theorem~\ref{T_finitely_generated}, Corollary~\ref{C_evar_max_finite}, Theorem~\ref{T_finite_constraint_set_reference_measure}, Theorem~\ref{T_sub_psi_rep}, and Theorem~\ref{T_symmetric_rep}, lead to conditions for admissibility. We now elaborate on this aspect of our results and show that we actually obtain \emph{minimal complete classes}, defined below, for the hypotheses under consideration. This generalizes \citet[Section~3]{clerico2024optimalevaluetestingproperly}.

\begin{definition} \label{D_complete_class}
A \emph{complete class} of e-variables for $\Pcal$ is a subset $\Ecal' \subset \Ecal$ such that every $h \in \Ecal$ is $\Pcal$-q.s.\ dominated by some $h' \in \Ecal'$. A complete class $\Ecal'$ is \emph{minimal} if removing an e-variable (and all its $\Pcal$-versions) from $\Ecal'$ renders the class non-complete.
\end{definition}

For statistical applications it suffices to work with complete classes, and minimal complete classes are preferable. Clearly $\Ecal$ itself is a complete class, but it is not minimal. Existence of a minimal complete class has some useful implications, recorded in the following result which extends \citet[Lemma~3.3]{clerico2024optimalevaluetestingproperly} to our more general class of hypotheses.

\begin{lemma} \label{L_minimal_complete_class}
A minimal complete class $\Ecal'$ exists if and only if every e-variable is $\Pcal$-q.s.\ dominated by an admissible e-variable. In this case, every $h' \in \Ecal'$ is admissible, and every admissible $h \in \Ecal$ has a version $h' \in \Ecal'$. In particular, a minimal complete class is unique provided different $\Pcal$-versions of any e-variable are identified.
\end{lemma}

\begin{proof}
We drop the ``$\Pcal$-q.s.''~qualifier for brevity. If every e-variable is dominated by an admissible e-variable, then $\Ecal' = \{h' \in \Ecal \colon h' \text{ is admissible}\}$ is a complete class. It is also minimal, simply because admissible e-variables cannot be strictly dominated. Conversely, if a minimal complete class $\Ecal'$ exists, then by completeness, every $h \in \Ecal$ is dominated by some $h' \in \Ecal'$, but by minimality, $h'$ cannot be further strictly dominated, so it is admissible. This reasoning also yields the remaining statements.
\end{proof}

Any hypothesis with a reference measure admits a minimal complete class. This extends \citet[Lemma~3.6]{clerico2024optimalevaluetestingproperly}, with a similar proof.

\begin{theorem}
Suppose $\Pcal$ admits a reference measure $\bar\mu \in \Mcal_1$, i.e., $\mu \ll \bar\mu$ for all $\mu \in \Pcal$. Then a minimal complete class exists.
\end{theorem}
 
\begin{proof}
We use an argument from \cite{ramdas2022admissibleanytimevalidsequentialinference} based on transfinite induction to show that every e-variable is dominated by an admissible e-variable.
The result then follows from Lemma~\ref{L_minimal_complete_class}. We pick any $h \in \Ecal$ and recursively define an increasing transfinite sequence of e-variables $h_\beta \in \Ecal$ indexed by the countable ordinals $\beta$. 
First, $h_0 = h$. For a successor ordinal $\beta = \alpha + 1$ such that $h_\alpha$ has already been defined, if $h_\alpha$ is admissible we set $h_\beta = h_\alpha$, otherwise we choose any $h_\beta \in \Ecal$ that strictly dominates $h_\alpha$. (I.e., $h_\beta \ge h_\alpha$, $\bar\mu$-a.s., and $h_\beta > h_\alpha$ with positive $\bar\mu$-probability.)
For any countable limit ordinal $\beta = \lim_n \alpha_n$ such that $h_{\alpha} \in \Ecal$ has been defined for all $\alpha < \beta$, we let $h_\beta = \lim_n h_{\alpha_n}$. By the monotone convergence theorem, $h_\beta \in \Ecal$. By transfinite induction, $h_\beta \in \Ecal$ for every countable ordinal $\beta$.
Consider now the decreasing $[0,1]$-valued transfinite sequence $c_\beta = \int_\Xcal e^{-h_\beta} d\bar\mu$.
This sequence must become constant for
all $\beta$ beyond some countable ordinal $\beta_0$; otherwise we would have an uncountable strictly decreasing sequence in $[0,1]$, which is impossible. By construction, $h_{\beta_0}$ is then admissible and dominates $h$.
\end{proof}

Unfortunately, as the following example shows, there are hypotheses for which a minimal complete class does not exist. A different example, involving a certain non-measurable set, is given by \citet[Appendix~C]{clerico2024optimalevaluetestingproperly}.

\begin{example}
Let $\Xcal = [0,1]$ and let $\Pcal$ consist of the standard uniform distribution together with all measures of the form $(\delta_0 + \delta_x)/2$ for $x \in [0,1]$. Then $\Ecal$ consists of all nonnegative measurable functions $h$ with $\int_0^1 h(x) dx \le 1$ and $h(0) + h(x) \le 2$ for all $x \in [0,1]$. Note that $x = 0$ is included here, so that $h(0) \le 1$.

We claim that if $h$ is an e-variable with $h(0) < 1$, then it cannot be admissible. To show this, assume $h(0) < 1$ and let $y \in (0,1]$ be such that $h(y) < 2 - h(0)$; such $y$ must exist, because otherwise $\int_0^1 h(x) dx = 2 - h(0) > 1$, contradicting the e-variable property. Then, the function $g$ with $g(y) = 2 - h(0)$ and $g = h$ elsewhere is still an e-variable, and it strictly dominates $h$ under $(\delta_0 + \delta_y)/2$. So $h$ is not admissible.

We deduce from this that the only way an e-variable $h$ can be admissible is if $h(0) = 1$. But then $h(x) \le 2 - h(0) = 1$ for all $x \in (0,1]$ so that, as a result, the only admissible e-variable is the trivial e-variable equal to one. Thus if $h(x) > 1$ for some $x \in (0,1]$ (e.g.\ $h(x) = 2$ and zero otherwise) then $h$ cannot be admissible, and it cannot be dominated by an admissible e-variable. Consequently, due to Lemma~\ref{L_minimal_complete_class}, no minimal complete class exists for this hypothesis.
\end{example}

In view of this example, it is remarkable that the hypotheses considered in Sections~\ref{S_fin_gen_hyp}--\ref{S_group_symmetry} do admit minimal complete classes under mild additional assumptions. (We conjecture that minimal complete classes exist even without these assumptions.)

\begin{theorem} \label{T_minimal_complete_class}
\begin{enumerate}
\item\label{T_minimal_complete_class_1} With the setup and notation of Section~\ref{S_fin_gen_hyp}, the set of nonnegative functions of the form $1 + \sum_{i=1}^d \pi_i g_i$, with $\pi \in \R^d_+$, is a complete class for the hypothesis generated by $g_1,\ldots,g_d$. If \eqref{C_max_evar_CQ} holds, this is a minimal complete class.

\item\label{T_minimal_complete_class_2}
With the setup and notation of Section~\ref{S_sub_psi}, let $\Ecal'$ be the set of functions of the form $\int_\Lambda h_\lambda(x) \pi(d\lambda)$ for $x \in \dom(\psi^*)$, where $h_\lambda(x) = e^{\lambda x - \psi(\lambda)}$ and $\pi \in \Mcal_1(\Lambda)$.\footnote{Recall that $\Lambda = [0,\lambda_\textnormal{max}]$ and that $h_{\lambda_\textnormal{max}}(x) = \lim_{\lambda \uparrow \lambda_\textnormal{max}} e^{\lambda x - \psi(\lambda)}$ for $x \in \Xcal$, where $\Xcal = \dom(\psi^*)$.}
Then $\Ecal'$ is a complete class for the one-sided sub-$\psi$ hypothesis.
Assume now that on $\R_+$, $\psi$ is the cumulant generating function of some zero-mean distribution $\bar\mu$ whose support has a cluster point in the interior of $\dom(\psi^*)$.
There are two cases depending on whether $\dom(\psi^*)$ contains its right endpoint:
\begin{enumerate}
\item[\textnormal{(1)}] if $\dom(\psi^*) = (-\infty, \bar x]$ with $\bar x \in [0,\infty)$, then $\Ecal'$ is a minimal complete class.
\item[\textnormal{(2)}] if $\dom(\psi^*) = (-\infty, \bar x)$ with $\bar x \in [0,\infty]$, then $\Ecal'_0$ is a minimal complete class, where $\Ecal'_0 \subset \Ecal'$ is the set of all functions of the form $\int_{\Lambda_0} h_\lambda(x) \pi(d\lambda)$ for $x \in \dom(\psi^*)$, with $\pi \in \Mcal_1(\Lambda_0)$ and $\Lambda_0 = [0, \lambda_\textnormal{max})$.
\end{enumerate}

\item\label{T_minimal_complete_class_3} With the setup and notation of Section~\ref{S_group_symmetry}, the set of e-variables of the form $1 + f - f_\pi$, with $f$ being $[-1,\infty]$-valued measurable and such that $f_\pi \le 0$, is a minimal complete class for the hypothesis of $\Sigma$-invariant distributions.
\end{enumerate}
\end{theorem}

\begin{proof}
\ref{T_minimal_complete_class_1}:
This follows directly from Corollary~\ref{C_evar_max_finite} (or Theorem~\ref{T_finite_constraint_set_reference_measure} in the setting with a reference measure) and Lemma~\ref{L_minimal_complete_class}.

\ref{T_minimal_complete_class_2}: Theorem~\ref{T_sub_psi_rep}\ref{T_sub_psi_rep_2} yields that $\Ecal'$ is a complete class. Assume now that $\psi$ is the cumulant generating function of some distribution $\bar \mu$ as in the statement.
We first consider case~(1). The minimality of $\Ecal'$ will follow from Lemma~\ref{L_minimal_complete_class} once we show that all e-variables in $\Ecal'$ are admissible.
To this end, pick $h \in \Ecal'$ specified by some $\pi \in \Mcal_1(\Lambda)$, meaning that $h(x) = \int_\Lambda h_\lambda(x) \pi(d\lambda)$ for $x \in \dom(\psi^*)$. Let $h'$ be an e-variable that dominates $h$.
By completeness, $h'$ is further dominated by some $h'' \in \Ecal'$, specified by some $\pi'' \in \Mcal_1(\Lambda)$.
We thus have $h'' - h \ge 0$ on $\dom(\psi^*)$. We claim that $\int_\Xcal (h'' - h) d\bar\mu = 0$. To see this, note that $\psi(\lambda)$ is the cumulant generating function of $\bar\mu$ at each $\lambda \in [0,\lambda_\textnormal{max})$, so that $\int_\Xcal h_\lambda(x) \bar\mu(dx) = 1$. 
Moreover, because $h_\lambda(x) \le h_\lambda(\bar x) \le e^{\psi^*(\bar x)} < \infty$ for $x \in \dom(\psi^*)$ and $\lambda \in [0,\lambda_\textnormal{max})$, the dominated convergence theorem yields $\int_\Xcal h_{\lambda_\textnormal{max}}(x) \bar\mu(dx) = 1$ as well. Thus by Tonelli's theorem,
\[
\int_\Xcal (h''(x) - h(x)) \bar\mu(dx) = \int_\Lambda \int_\Xcal h_\lambda(x) \bar\mu(dx) (\pi'' - \pi)(d\lambda) = \pi''(\Lambda) - \pi(\Lambda) = 0.
\]
Hence $h''(x) = h(x)$ for $\bar\mu$-a.e.\ $x$, and then, by continuity, for all $x \in \supp(\bar\mu)$. Since $h$ and $h''$ are real analytic in $(-\infty, \bar x)$, and since $\supp(\bar\mu)$ has a cluster point there, the identity theorem for real analytic functions (see e.g.\ \citet[Corollary~1.2.7]{MR1916029}) implies that $h'' = h$ on $(-\infty, \bar x)$ and then by continuity on $\dom(\psi^*)$. It follows that $h' = h$ on $\dom(\psi^*)$, showing that $h$ is admissible.

Consider now case~(2). Then $h_{\lambda_\textnormal{max}}(x) = 0$ for all $x \in \dom(\psi^*)$. Indeed, if $\lambda_\textnormal{max} < \infty$, this follows from the fact that $\psi(\lambda) \to \infty$ as $\lambda \uparrow \lambda_\textnormal{max}$ by virtue of being a cumulant generating function. If $\lambda_\textnormal{max} = \infty$, we pick $y \in (x, \bar x)$ and note that $h_\lambda(x) \le e^{-\lambda(y-x) + \psi^*(y)}$, which tends to zero as $\lambda \to \infty$. 
As a result, any $h \in \Ecal'$, specified by $\pi \in \Mcal_1(\Lambda)$, is dominated by $h' \in \Ecal'_0$ specified by $\pi' = \pi(\cdot \cap \Lambda_0) + (1-\pi(\{\lambda_{\textnormal{max}}\})) \delta_0 \in \Mcal_1(\Lambda_0)$. 
Thus $\Ecal'_0$ is a complete class. The proof of minimality works as in case~(1) on replacing $\Ecal$ by $\Ecal_0$ and $\Lambda$ by $\Lambda_0$. (Note that since $\Lambda_0$ does not contain $\lambda_\textnormal{max}$, there is no need to establish $\int_\Xcal h_{\lambda_\textnormal{max}}(x) \bar\mu(dx) = 1$, which does not hold in the setting of case~(2).)

\ref{T_minimal_complete_class_3}: This follows directly from Theorem~\ref{T_symmetric_rep} and Lemma~\ref{L_minimal_complete_class}, noting that exact e-variables (see Theorem~\ref{T_symmetric_rep}) are necessarily admissible.
\end{proof}

\section{Optimal e-variables} \label{S:optimal}

We now consider the problem of finding \emph{optimal} e-variables in the set $\Ecal$ of all e-variables for a given hypothesis $\Pcal$. 
We work with a general class of objective functions defined in terms of the following data:
\begin{itemize}
\item \emph{Alternative hypothesis}: any family $\Qcal \subset \Mcal_1$ of probability measures such that $\Qcal \ll \Pcal$, meaning that every $\Pcal$-negligible set is also $\Qcal$-negligible.

\item \emph{Utility function}: a nondecreasing upper semicontinuous function $U \colon [0,\infty) \to [-\infty, \infty)$ that is linearly bounded in the sense that there exist $a \in \R$ and $b > 0$ such that
\[
U(x) \le a + bx \text{ for all } x \in [0,\infty).
\]

\item \emph{Penalty function}: an arbitrary function $\chi \colon \Qcal \to \R$.
\end{itemize}
The objective function is defined in terms of these data as
\[
\U(h) = \inf_{\Q \in \Qcal} \left\{ \int_\Xcal U(h) d\Q - \chi(\Q)\right\}, \quad h \in \Ecal,
\]
with the convention that the integral is $-\infty$ whenever $\int_\Xcal U(h)^- d\Q = \infty$.
This covers several criteria appearing in the literature:

\begin{example}
\begin{enumerate}
\item With a simple hypothesis $\Qcal = \{\Q\}$, log-utility $U(x) = \log(x)$, and vanishing penalty function $\chi=0$, one 
maximizes \emph{e-power} \citep[Ch.~3]{ramdas2024hypothesis}, thus  obtaining (assuming the optimal value is finite) the log-optimal \emph{GRO e-variable} or \emph{numeraire} of \cite{grunwald2020safe,larsson2024numeraire}.

\item More generally, log-utility can be replaced with any proper concave function $U$. By concavity, such a function is automatically linearly bounded. Examples include the power utilities considered by \citet[Section~6]{larsson2024numeraire} and those used by \citet{koning2025continuoustestingunifyingtests}.

\item If, instead, we keep log-utility, let $\Qcal$ be composite, assume that both $\Pcal$ and $\Qcal$ admit a common reference measure, and maintain a zero penalty function, we obtain the \emph{GROW e-variable} problem of \cite{grunwald2020safe}. Of course, we can again work with general concave utility functions.

\item Still in the composite setting with reference measure and log-utility $U(x) = \log(x)$, allowing a bounded penalty function $\chi$ yields the \emph{REGROW e-variable} problem of \cite{grunwald2020safe}. The main example is $\chi(\Q) = \sup_{h \in \Ecal} \int_\Xcal U(h) d\Q$, the optimal e-power for the point alternative $\Q$. The objective function $\U(h)$ can then be interpreted as (the negative of) the worst-case \emph{regret} of using $h$ instead of the e-variable with optimal e-power for the particular alternative $\Q \in \Qcal$ that actually generated the data. Note that our setup does not require the penality function $\chi$ to be bounded.

\item Any \emph{continuous variational preference} in the sense of \cite{MR2268407} is represented by a utility $\U(h)$ of the above form, with $U$ affine and $\chi$ a nonpositive function which, when extended to take the value $-\infty$ outside $\Qcal$, is concave and upper semicontinuous with weakly compact super-level sets; see \cite[Sections~3.2 and~3.4]{MR2268407} for details.
\end{enumerate}
\end{example}

Because $U$ is nondecreasing, so is $\U$ in the sense that $\U(h') \ge \U(h)$ whenever $h' \ge h$, $\Pcal$-q.s.\ and hence $\Qcal$-q.s. For this reason it suffices to maximize $\U$ over a complete class of e-variables; see Definition~\ref{D_complete_class}.
Several of the results in this paper yield complete classes $\Ecal_0$ which are \emph{compactly parameterized} in the sense that there is a compact set $K_0$ and a continuous map $\pi \mapsto h(\pi)$ from $K_0$ onto $\Ecal_0$. 
The problem of maximizing $\U(h)$ over $\Ecal$ thus reduces to maximizing $\U(h(\pi))$ over $K_0$. 
Thanks to this reduction, the existence of an optimal e-variable boils down to the standard fact that an upper semicontinuous function on a compact set achieves its supremum. In contrast, maximizing $\U(h)$ directly over $\Ecal$ seems intractable in general, especially when $\Qcal$ is non-dominated or $U$ is non-concave. We elaborate below on the details of this reduction in two cases: finitely generated hypotheses and one-sided sub-$\psi$ hypotheses. After that we discuss uniqueness of the optimizer.

\paragraph{Finitely generated hypotheses.}
Consider a finite constraint set $\Phi = \{g_1,\ldots,g_d\}$ and let $\Pcal$ be the hypothesis it generates. We know from Theorem~\ref{T_minimal_complete_class}\ref{T_minimal_complete_class_1}, or directly from Theorem~\ref{T_finitely_generated}, that a complete class is given by all e-variables of the form
\[
h = 1 + \sum_{i=1}^d \pi_i g_i
\]
with $\pi \in \R^d_+$ belonging to the set $\Pi^\Phi$ defined in Corollary~\ref{C_evar_max_finite}. Thanks to Lemma~\ref{L_not_redundant} we may always restrict $\pi$ to the set $K$ in \eqref{eq_nonredundant_K}. Defining $K_0 = K \cap \Pi^\Phi$, the problem of maximizing $\U(h)$ over $\Ecal$ thus reduces to
\begin{equation} \label{eq_u_of_pi_finite}
\sup_{\pi \in K_0} u(\pi), \quad \text{where} \quad u(\pi) = \inf_{\Q \in \Qcal} \left\{ \int_\Xcal U\left(1 + \sum_{i=1}^d \pi_i g_i\right) d\Q - \chi(\Q) \right\}.
\end{equation}
The key advantage of this reduction is captured by the following lemma.

\begin{lemma} \label{L_K0_compact_finite}
The set $K_0 = K \cap \Pi^\Phi$ is compact.
\end{lemma}

\begin{proof}
Both $K$ and $\Pi^\Phi$ are closed, so it suffices to prove that $K_0$ is bounded. The argument is similar to the proof of Theorem~\ref{T_finitely_generated}. 
Suppose for contradiction that there is a sequence $(\pi^n)_{n \in \N}$ in $K_0$ with $\|\pi^n\| \to \infty$, where $\|\pi^n\| = \sum_{i=1}^d \pi^n_i$. 
After passing to a subsequence and using that $K_0$ is closed and contains $\pi^n / \|\pi^n\|$ whenever $\|\pi^n\| \ge 1$, we obtain $\pi^n/\|\pi^n\| \to \rho$ for some $\rho \in K_0$ with $\|\rho\|=1$. 
We then get $\sum_{i=1}^d \rho_i g_i = \lim_{n \to \infty} (1+\sum_{i=1}^d \pi^n_i g_i)/\|\pi^n\| \ge 0$, $\Pcal$-q.s. On the other hand, $\int_\Xcal \sum_{i=1}^d \rho_i g_i d\mu \le 0$ for all $\mu \in \Pcal$. Thus $\sum_{i=1}^d \rho_i g_i = 0$, $\Pcal$-q.s., contradicting the fact that $\rho \in K$. We deduce that $K_0$ is bounded.
\end{proof}

It is now straightforward to show existence of an optimal e-variable.

\begin{theorem} \label{T_opt_evar_finite}
Assume the constraint set is finite, $\Phi = \{g_1,\ldots,g_d\}$, and that
\[
\int_\Xcal \left( |g_1| + \cdots + |g_d| \right) d\Q < \infty \text{ for all } \Q \in \Qcal.
\]
Then there exists an e-variable $h^* = 1 + \sum_{i=1}^d \pi^*_i g_i$ with $\pi^* \in K_0$ that achieves the supremum $\sup_{h \in \Ecal} \U(h)$.
\end{theorem}

\begin{proof}
We need to find $\pi^* \in K_0$ that maximizes $u$ in \eqref{eq_u_of_pi_finite}. Since $K_0$ is compact by Lemma~\ref{L_K0_compact_finite}, it suffices to show that $u$ is upper semicontinuous. Set $c = \sup\{\|\pi\| \colon \pi \in K_0\}$, which is finite since $K_0$ is compact, and note that $U(1 + \sum_{i=1}^d \pi_i g_i) \le a + b + b c (|g_1| + \cdots |g_d|)$ for any $\pi \in K_0$. Since the upper bound is integrable under any fixed $\Q \in \Qcal$, it follows from Fatou's lemma and upper semicontinuity of $U$ that the map $\pi \mapsto \int_\Xcal U(1 + \sum_{i=1}^d \pi_i g_i) d\Q - \chi(\Q)$ from $K_0$ to $[-\infty,\infty]$ is upper semicontinuous.
Thus $u$ is upper semicontinuous too, being the pointwise infimum of a family of upper semicontinuous functions.
\end{proof}

\paragraph{One-sided sub-$\psi$ hypotheses.}
We now consider the hypothesis $\Pcal$ associated with a function $\psi$ as in Section~\ref{S_sub_psi}.
We then get from Theorem~\ref{T_minimal_complete_class}\ref{T_minimal_complete_class_2}, or directly from Theorem~\ref{T_sub_psi_rep}\ref{T_sub_psi_rep_2}, that a complete class is given by all e-variables of the form
\[
h(x) = \int_\Lambda e^{\lambda x - \psi(\lambda)} \pi(d\lambda), \quad x \in \dom(\psi^*),
\]
with $\pi \in \Mcal_1(\Lambda)$. The problem of maximizing $\U(h)$ over $\Ecal$ thus reduces to maximizing $u(\pi)$ over $\Mcal_1(\Lambda)$, where
\begin{equation} \label{eq_u_of_pi_sub_psi}
u(\pi) = \inf_{\Q \in \Qcal} \left\{ \int_\Xcal U\left(\int_\Lambda e^{\lambda x - \psi(\lambda)} \pi(d\lambda)\right) d\Q - \chi(\Q) \right\}.
\end{equation}
Just as in the finitely generated case, this reduction immediately leads to a general existence result.

\begin{theorem} \label{T_optimal_sub_psi}
Assume $\Pcal$ is the one-sided sub-$\psi$ hypothesis associated with a function $\psi$ as in Section~\ref{S_sub_psi}, and that
\[
\int_\Xcal e^{\psi^*} d\Q < \infty \text{ for all } \Q \in \Qcal.
\]
Then there exists an e-variable $h^*(x) = \int_\Lambda e^{\lambda x - \psi(\lambda)} \pi^*(d\lambda)$ with $\pi^* \in \Mcal_1(\Lambda)$ that achieves the supremum $\sup_{h \in \Ecal} \U(h)$.
\end{theorem}

\begin{proof}
Since $\Lambda = [0, \lambda_\textnormal{max}]$ is compact, so is $\Mcal_1(\Lambda)$ thanks to Prokhorov's theorem; see, e.g., \citet[Theorem~5.1]{bil_99}. Here $\Mcal_1(\Lambda)$ is equipped with the topology of weak convergence. Moreover, if a sequence $(\pi_n)_{n \in \N}$ in $\Mcal_1(\Lambda)$ converges weakly to a limit $\pi$, then the integrals $\int_\Lambda e^{\lambda x - \psi(\lambda)} \pi_n(d\lambda)$ converge to $\int_\Lambda e^{\lambda x - \psi(\lambda)} \pi(d\lambda)$ for every $x \in \dom(\psi^*)$. Since 
\[
U\left(\int_\Lambda e^{\lambda x - \psi(\lambda)} \pi(d\lambda)\right) \le a + b e^{\psi^*(x)}
\] for all $\pi \in \Mcal_1(\Lambda)$, and since the upper bound is integrable under any fixed $\Q \in \Qcal$, the same argument as in the proof of Theorem~\ref{T_opt_evar_finite} now yields upper semicontinuity of $u$ in \eqref{eq_u_of_pi_sub_psi}. It follows that an optimal e-variable $h^*$ of the stated form exists.
\end{proof}

Here is an interesting example where the optimal $h^*$ and $\pi^*$ in Theorem~\ref{T_optimal_sub_psi} are explicit.
We thank one of the referees for suggesting this example.

\smallskip\begin{example}   
Fix a nondegenerate mean-zero distribution $\mu_0$ on $\Xcal = \R$ with finite cumulant generating function $\psi(\lambda)$, $\lambda \in [0,\infty)$.
We let the null hypothesis $\Pcal$ be the one-sided sub-$\psi$ hypothesis.
Next, the exponential family generated by $\mu_0$ consists of the distributions $\mu_\lambda(dx) = \exp(\lambda x - \psi(\lambda)) \mu_0(dx)$, parameterized by $\lambda \in [0,\infty)$.
Fix some $\lambda^* > 0$ and consider the simple alternative $\Q^* = \mu_{\lambda^*}$. Note that $\Q^*$ is not sub-$\psi$, because for $\lambda > 0$,
\[
\int_\Xcal e^{\lambda x - \psi(\lambda)} \Q^*(dx) = \int_\Xcal e^{(\lambda + \lambda^*) x - \psi(\lambda) - \psi(\lambda^*)} \mu_0(dx) = \e^{\psi(\lambda + \lambda^*) - \psi(\lambda) - \psi(\lambda^*)} > 1
\]
by super-additivity of $\psi$ for positive arguments.
Now, because the function $h^*(x) = \exp(\lambda^* x - \psi(\lambda^*))$ is an e-variable by definition of $\Pcal$ and, also, the likelihood ratio between $\Q^*$ and an element $\mu_0$ of $\Pcal$, $h^*$ is actually the numeraire, and as such optimal for the log-utility $U(x) = \log(x)$; see \citet[Corollary~2]{grunwald2020safe} or \citet[Theorem~4.1]{larsson2024numeraire}.
This identifies $h^*$ and $\pi^* = \delta_{\lambda^*}$ in this case.

Remarkably, the same $h^*$ remains optimal in the sense of GROW if the above simple alternative is replaced by the composite alternative $\Qcal = \{\Q \colon \int_\Xcal x^- \Q(dx) < \infty \text{ and } \int_\Xcal x \Q(dx) \ge m^*\}$, where $m^* = \int_\Xcal x \Q^*(dx) = \psi'(\lambda^*)$. In words, $\Qcal$ consists of all distributions with integrable left tail whose (possibly infinite) mean is at least as large as the mean of $\Q^*$. Because $\psi'$ is increasing, $\Qcal$ contains in particular all members $\mu_\lambda$ of the exponential family for which $\lambda \ge \lambda^*$.
To see that $h^*$ is GROW optimal, first note that any $\Q \in \Qcal$ satisfies
\[
\int_\Xcal \log(h^*(x)) \Q(dx) = \lambda^* \int_\Xcal x \Q(dx) - \psi(\lambda^*) \ge \lambda^* m^* - \psi(\lambda^*) = \int_\Xcal \log(h^*(x)) \Q^*(dx).
\]
On the other hand, since $h^*$ is log-optimal for the simple alternative $\{\Q^*\}$, the right-hand side dominates $\int_\Xcal \log( h(x)) \Q^*(dx)$ for any other e-variable $h$.
This yields the saddle-point property
\[
\int_\Xcal \log(h^*) d\Q \ge \int_\Xcal \log(h^*) d\Q^* \ge \int_\Xcal \log(h) d\Q^*
\]
for all $\Q \in \Qcal$ and $h \in \Ecal$. The claim that $h^*$ is GROW follows.
\end{example}

\paragraph{Uniqueness.}
Under suitable strict concavity conditions one expects that the e-variable attaining the supremum is unique. However, since the objective value $\U(h)$ is unchanged if $h$ is replaced by a $\Qcal$-version, one can only expect uniqueness up to $\Qcal$-negligible sets. In what follows, uniqueness will always be understood in this sense. It is clear that if $U$ is concave, then so is $\U$. We say that $\U$ is \emph{strictly concave} if $\U(th + (1-t)h') > t\U(h) + (1-t)\U(h')$ holds for all $t \in (0,1)$ and all $h,h' \in \Ecal$ that are not $\Qcal$-q.s.\ equal. Strict concavity of $\U$ implies uniqueness of the optimal e-variable. However, in the general composite case, strict concavity of $\U$ does not follow from strict concavity of $U$. We now present a simple sufficient condition for uniqueness in terms of $U$ and $\Qcal$.

\begin{theorem} \label{T_optimal_uniqueness}
Assume that $U$ is concave, $\sup_{h \in \Ecal} \U(h) \in \R$, and the following condition holds:
\[
\text{if $h_1,h_2 \in \Ecal$ and } \inf_{\Q \in \Qcal} \int_\Xcal \left( U\left(\frac{h_1 + h_2}{2}\right) - \frac{U(h_1)}{2} - \frac{U(h_2)}{2} \right) d\Q = 0, \text{ then $h_1 = h_2$, $\Qcal$-q.s.}
\]
Then the optimal e-variable is unique. The above condition holds, in particular, if $U$ is strictly concave and $\Qcal$ is contained in the convex hull of finitely many mutually absolutely continuous probability measures. 
\end{theorem}

\begin{proof}
Let $h_1, h_2 \in \Ecal$ be any optimal e-variables, and set $\overline h = (h_1 + h_2)/2$. We then have
\begin{align*}
\U(\overline h) &= \inf_{\Q \in \Qcal} \left\{ \int_\Xcal U(\overline h) d\Q - \chi(\Q) \right\} \\
&= \inf_{\Q \in \Qcal} \bigg\{ \frac12 \left( \int_\Xcal U(h_1) d\Q - \chi(\Q) \right) + \frac12 \left( \int_\Xcal U(h_2) d\Q - \chi(\Q) \right) \\
&\qquad\qquad + \int_\Xcal \left(U(\overline h) - \frac12 U(h_1) - \frac12 U(h_2)\right) d\Q \bigg\} \\
&\ge \frac12\U(h_1) + \frac12\U(h_2) + \inf_{\Q \in \Qcal} \int_\Xcal \left(U(\overline h) - \frac12 U(h_1) - \frac12 U(h_2)\right) d\Q.
\end{align*}
The first line is just the definition of $\U(\overline h)$.
To justify the second line, note that concavity of $U$ yields $U(\overline h) - U(h_1)/2 - U(h_2)/2 \ge 0$, and that the negative parts $U(h_i)^-$, $i=1,2$, are integrable under every $\Q \in \Qcal$ since $\sup_{h \in \Ecal} \U(h) > -\infty$. The last line follows from the definition of $\U(h_i)$, $i=1,2$. We deduce that $\overline h$ is optimal, and then, since the optimal value is finite, that
\[
\inf_{\Q \in \Qcal} \int_\Xcal \left( U(\overline h) - \frac12 U(h_1) - \frac12 U(h_2)\right) d\Q = 0.
\]
The assumption of the theorem now yields $h_1=h_2$, $\Qcal$-q.s., showing uniqueness.
\end{proof}

\section{Hypotheses expressed as unions} \label{S_unions}

Some hypotheses of interest are naturally expressed as unions of simpler hypotheses. This happens, for instance, if one imposes an upper bound on the \emph{conditional value-at-risk}; see Example~\ref{ex_es_bound} below. Hypotheses of this form are not generated by constraints in a natural way, but our results can still sometimes be used to describe their e-variables. This is a consequence of the following simple but very general result.

\begin{lemma} \label{L_union_hypotheses}
Consider an arbitrary collection of hypotheses $\Pcal_\theta$, $\theta \in \Theta$, and let $\Ecal_\theta$ be the set of e-variables for $\Pcal_\theta$. Define the union $\Pcal = \bigcup_{\theta \in \Theta} \Pcal_\theta$, and let $\Ecal$ be the associated set of e-variables. Then $\Ecal = \bigcap_{\theta \in \Theta} \Ecal_\theta$.
\end{lemma}

\begin{proof}
A measurable function $h \colon \Xcal \to [0,\infty]$ belongs to $\bigcap_{\theta \in \Theta} \Ecal_\theta$ if and only if $\int_\Xcal h d\mu \le 1$ for all $\mu \in \Pcal_\theta$ and all $\theta \in \Theta$. But this just says that $f \in \Ecal$.
\end{proof}

It is worth mentioning that there is no analog of Lemma~\ref{L_union_hypotheses} allowing to reconstruct the set $\Ecal$ of e-variables for an \emph{intersection} $\Pcal = \bigcap_{\theta \in \Theta} \Pcal_\theta$ from the individual sets $\Ecal_\theta$. Such an analog would be of interest, because any hypothesis generated by a constraint set $\Phi$ can be expressed as the intersection $\bigcap_{f \in \Phi} \Pcal_f$, where $\Pcal_f$ is the hypothesis generated by the single function $f$. To see the obstruction, consider the constraint function $f(x) = x$. We know from Example~\ref{ex_unbounded_mean} that $\Ecal_{f}$ is trivial, and thus contains no information about the constraint function that generated it.
It follows that there is no way to use $\Ecal_f$ and, say, $\Ecal_g$ where $g(x) = x^2 - 1$, to construct the set of e-variables for the hypothesis generated by $f$ \emph{and} $g$, whose e-variables are characterized in Example~\ref{ex:250302}.

We now illustrate how Lemma~\ref{L_union_hypotheses} can be used by examining a hypothesis involving a bound on conditional value-at-risk (CVaR). Interestingly, this example shows that e-variables can be used to test CVaR without imposing any moment bounds as is done, for instance, by \cite{NEURIPS2021_d69c7ebb}. Backtesting of CVaR using e-variables has been studied by \citet{wang2024ebacktesting}. The representation derived below extends the one in their Theorem~5 by removing the restriction to e-variables that are increasing in $x$ imposed in that theorem.

\begin{example} \label{ex_es_bound}
Let $\Xcal = \R$. 
For any $\alpha \in (0,1)$,
the \emph{conditional value-at-risk} at level $\alpha$ of a distribution $\mu$ is given by
\begin{equation} \label{eq_es_rep}
\textnormal{CVaR}_\alpha(\mu) = \min_{\theta \in \R} \left\{\theta + \frac{1}{1-\alpha}\int_\Xcal (x-\theta)^+ \mu(dx) \right\}.
\end{equation}
This formula can be interpreted in terms of the \emph{value-at-risk} of $\mu$ at level $\alpha$, which is the quantile $\textnormal{VaR}_\alpha(\mu) = \min\{x \in \R \colon \mu((-\infty,x]) \ge \alpha\}$. Assuming that $\mu$ does not have an atom at $\textnormal{VaR}_\alpha(\mu)$, one has
$\textnormal{CVaR}_\alpha(\mu) = \E[X \mid X \ge \textnormal{VaR}_\alpha(\mu)]$ in probabilistic notation, where $X$ is a random variable with distribution $\mu$. A slightly more involved formula holds in the general case. We refer to \citet[Definition~3]{ROCKAFELLAR20021443} for details.

For a given level $\alpha \in (0,1)$ and threshold $c \in \R$, we consider the hypothesis that the conditional value-at-risk at level $\alpha$ does not exceed $c$, i.e.,
\[
\Pcal = \{\mu \in \Mcal_1 \colon \textnormal{CVaR}_\alpha(\mu) \le c\}.
\]
In view of \eqref{eq_es_rep}, $\Pcal$ is the union of the hypotheses $\Pcal_\theta$, each generated by the single constraint function $\theta - c + (x-\theta)^+/(1-\alpha)$. 
For $\theta > c$, this function is strictly positive, so $\Pcal_\theta$ is empty and can be excluded from the union. We thus restrict to $\theta \le c$. 
For $\theta = c$, $\Pcal_\theta$ consists of all distributions on $(-\infty,c]$, and admits the single admissible e-variable $1 + \infty \bm1_{(c,\infty)}(x)$. 
For $\theta < c$, we get from Corollary~\ref{C_evar_max_finite} that every e-variable for $\Pcal_\theta$ is pointwise dominated by $1 + \pi_\theta (\theta - c + (x-\theta)^+/(1-\alpha))$ for some $\pi_\theta \in [0, (c-\theta)^{-1}]$. 
We then deduce from Lemma~\ref{L_union_hypotheses} that every e-variable for $\Pcal$ is dominated by a function of the form
\[
h(x) = 1 + \inf_{\theta < c} \pi_\theta \left(\theta - c + \frac{(x-\theta)^+}{1-\alpha}\right) \wedge \left( \infty \bm1_{(c,\infty)}(x) \right),
\]
for some $\pi_\theta \in [0, (c-\theta)^{-1}]$. 
Here the only role played by the term $\infty \bm1_{(c,\infty)}(x)$ is to ensure that $h(c) = 1$. 
Indeed, for $x < c$ the infimum is negative, and for $x > c$ it is, of course, finite. Thus the minimum with $\infty \bm1_{(c,\infty)}(x)$ has no effect in these cases. 
But it does have an effect for $x = c$, since $h(c) = 1 + \frac{\alpha}{1-\alpha} \inf_{\theta < c} \pi_\theta (c-\theta) \wedge 0 = 1$.
It is natural to reparameterize in terms of $w_\theta = \pi_\theta (c-\theta)$, and we finally arrive at the form
\begin{equation} \label{eq_h_cvar}
h(x) = 
\begin{cases}
\displaystyle 1 + \inf_{\theta < c} w_\theta \left(\frac{1}{1-\alpha}\left(\frac{x-\theta}{c - \theta}\right)^+ - 1\right), & x \ne c,\\
1, & x = c,
\end{cases}
\end{equation}
where $w_\theta \in [0,1]$.
Since the pointwise infimum of a family of continuous functions is upper semicontinuous and hence measurable, any $h$ of the form \eqref{eq_h_cvar} is measurable.
As a result, $h$ is itself an e-variable for $\Pcal$ and, therefore, the functions of this form constitute a complete class for $\Pcal$ in the sense of Definition~\ref{D_complete_class}.
\end{example}

It is remarkable that CVaR, unlike the mean, can be tested without any a priori moment bounds. Indeed, taking $c=0$ for simplicity, we obtain a nontrivial e-variable by setting $w_\theta = 1$ for all $\theta < 0$ in \eqref{eq_h_cvar}, namely
\[
h(x) = \begin{cases}
0, & x < 0, \\
1, & x = 0,\\
(1-\alpha)^{-1}, & x > 0.
\end{cases}
\]
Although $\alpha$ is required to lie in $(0,1)$ above, it is worth noting that with $\alpha=0$, CVaR formally reduces to the mean. In this case any choice of $w_\theta \in [0,1]$ in \eqref{eq_h_cvar} yields a trivial e-variable $h(x) \le 1$, which is consistent with Example~\ref{ex_unbounded_mean}.

\section{Hypotheses with relaxed integrability}
\label{S_relaxed_integrability}

We now elaborate on the integrability requirement in the definition of $\Pcal$ in \eqref{eq_hypothesis_integrable}. To illustrate the issue, consider a single constraint function $f_0$ which is bounded above but unbounded below. We can then find $\mu \in \Mcal_1$ such that $\int_\Xcal f_0 d\mu$ is well-defined but equal to $-\infty$. \emph{This measure does not qualify for membership in $\Pcal$.} More generally, given a general constraint set $\Phi$, it is natural to consider the larger, relaxed, hypothesis
\begin{equation} \label{eq_relaxed_hypothesis}
\widetilde\Pcal = \left\{\mu \in \Mcal_1 \colon \int_\Xcal f^+ d\mu < \infty \text{ and } \int_\Xcal f d\mu \in [-\infty,0] \text{ for all } f \in \Phi\right\}.
\end{equation}
What is the set $\widetilde\Ecal$ of e-variables for $\widetilde\Pcal$? It is certainly included in $\Ecal$, but can the two be different? Relatedly, suppose we start with a constraint set $\Phi$ and include a single additional \emph{negative} function $f_0$ to form $\Phi_0 = \Phi \cup \{f_0\}$. How, if at all, does this affect the hypotheses and their sets of e-variables? Intuitively one might expect that including a negative constraint function would not change the hypothesis. In the following discussion we indicate the constraint sets explicitly by writing $\Pcal(\Phi)$, $\Pcal(\Phi_0)$, $\widetilde\Pcal(\Phi)$, and $\widetilde\Pcal(\Phi_0)$ for the hypotheses, and $\Ecal(\Phi)$, etc., for the corresponding sets of e-variables.

One may wonder why the theory presented so far does not already operate with the relaxed definition in \eqref{eq_relaxed_hypothesis}. The reason is that $\widetilde\Pcal$ is in general not a subset of a vector space like $\Mcal^\Phi$. This renders the machinery of topological vector spaces and their duality theory inapplicable. However, once developed, our theory can be used to draw conclusions about the more general objects $\widetilde\Pcal$ and $\widetilde\Ecal$, as we now discuss.

It is clear that the relaxed hypotheses $\widetilde\Pcal(\Phi)$ and $\widetilde\Pcal(\Phi_0)$ are equal, and so are their sets of e-variables. However, as the following example shows, $\Pcal(\Phi)$ can differ from $\Pcal(\Phi_0)$, and their sets of e-variables can also be different.  Further, it shows $\widetilde\Ecal(\Phi_0) \neq \Ecal(\Phi_0)$ in general.

\begin{example}
Let $\Xcal = \N$. Consider the constraint set $\Phi$ consisting of the functions $f_n$, $n \in \N$, defined by $f_n(x) = 1$ for $x \ne n$ and $f_n(n) = 1-2^n$. Let $\Phi_0 = \Phi \cup \{f_0\}$, where the negative function $f_0$ is given by $f_0(x) = -2^x$. For any $\mu = \sum_{x \in \N} p_x \delta_x$ in the hypothesis $\Pcal(\Phi)$, the condition $\int_\N f_n d\mu \le 0$ translates to the inequality $1-p_n + p_n(1-2^n) \le 0$, or $p_n \ge 2^{-n}$. It follows that $\Pcal(\Phi)$ consists of the single measure $\mu = \sum_{x \in \N} 2^{-x}\delta_x$, and that $\Ecal(\Phi)$ consists of all nonnegative functions $h$ such that $\sum_{x \in \N} h(x) 2^{-x} \le 1$. However, $\int_\N f_0 d\mu = - \infty$, so $\mu\notin \Pcal(\Phi_0)$. Thus $\Pcal(\Phi_0) = \emptyset$, and $\Ecal(\Phi_0)$ consists of \emph{all} nonnegative functions on $\N$. In contrast, the relaxed hypotheses $\widetilde\Pcal(\Phi)$ and $\widetilde\Pcal(\Phi_0)$ both coincide with $\Pcal(\Phi)$, and their (common) set  of e-variables 
$\widetilde\Ecal(\Phi) = \widetilde\Ecal(\Phi_0) = \Ecal(\Phi)$.  Moreover, because these sets differ from $\Ecal(\Phi_0)$, we see that the relaxed definition \eqref{eq_relaxed_hypothesis} can indeed produce a strictly smaller set of e-variables.
\end{example}

In this example the constraint sets are infinite. We now give a positive result showing that with a finite constraint set, the above issue cannot occur. It is an interesting open problem to characterize $\widetilde\Ecal(\Phi)$ for a general infinite constraint set $\Phi$.

\begin{theorem} \label{T_Pcal_tilde_finite}
Consider a finite nonempty constraint set $\Phi = \{g_1,\ldots,g_d\}$. Let $\Pcal$ be the hypothesis it generates, and let $\widetilde\Pcal$ be the relaxed hypothesis defined in \eqref{eq_relaxed_hypothesis}. Then the set $\widetilde\Ecal$ of e-variables for $\widetilde\Pcal$ coincides with the set $\Ecal$ of e-variables for $\Pcal$.
\end{theorem}

The proof relies on the following lemma.

\begin{lemma} \label{L_barycenter}
Let $\mu \in \Mcal_1$ and let $f_1,\ldots,f_m$ be measurable and $\mu$-integrable. Then there is a finitely supported probability measure $\nu$ such that $\int_\Xcal f_i d\nu = \int_\Xcal f_i d\mu$ for $i=1,\ldots,m$.
\end{lemma}

\begin{proof}
Define $f\colon \Xcal \to \R^m$ by $f(x) = (f_1(x),\ldots,f_m(x))$ and set $z_0 = \int_\Xcal f d\mu \in \R^m$. To prove the lemma it is enough to show that $z_0 \in \conv(\range(f))$, the convex hull of the range of $f$. To this end, consider the pushforward $\gamma = f_*\mu$ on $\R^m$ and define
\[
C = \conv(\supp(\gamma) \cap \range(f)).
\]
Suppose for contradiction that $z_0 \notin \ri(C)$, the relative interior of $C$. Then there is some nonzero $u \in \R^m$ such that
\begin{equation} \label{eq_sep_hyperplane_ri_C}
\text{$u \cdot (z-z_0) > 0$ for all $z \in \ri(C)$.}
\end{equation} 
This implies that $u \cdot (z-z_0) \ge 0$ for all $z \in C$, hence $\gamma$-a.e. Since also 
\[
    \int_{\R^m} u \cdot (z-z_0) \gamma(dz) = u \cdot \left(\int_{\R^m} z \gamma(dz) - z_0\right) = 0, 
\]
we deduce that $u \cdot (z-z_0) = 0$, $\gamma$-a.e. This shows that $\supp(\gamma)$ is contained in the set $\{z \in \R^m \colon u\cdot(z - z_0) = 0\}$, which is disjoint from $\ri(C)$ in view of \eqref{eq_sep_hyperplane_ri_C}. We have established that $C \cap \ri(C) = \emptyset$, which contradicts the fact that every nonempty convex set in $\R^m$ has nonempty relative interior. Thus $z_0 \in \ri(C) \subset \conv(\range(f))$.
\end{proof}

\begin{proof}[Proof of Theorem~\ref{T_Pcal_tilde_finite}]
Since $\Pcal \subset \widetilde\Pcal$, we always have $\widetilde\Ecal \subset \Ecal$. We therefore only have to prove the opposite inclusion, so we fix any $h \in \Ecal$. Assume first that $h$ is bounded. Consider any $\mu \in \widetilde\Pcal$ and let $c$ be a sufficiently large negative constant to ensure that $\int_\Xcal (g_i \vee c) d\mu \le 0$ for those $i \in \{1,\ldots,d\}$ such that $\int_\Xcal g_i^- d\mu = \infty$. Using Lemma~\ref{L_barycenter} we obtain a finitely supported probability measure $\nu$ such that
\begin{align}
\int_\Xcal h d\nu &= \int_\Xcal h d\mu, && \label{eq_mu_F_1} \\
\int_\Xcal g_i d\nu &= \int_\Xcal g_i d\mu, && i \in \{1,\ldots,d\} \text{ with } \int_\Xcal g_i^- d\mu < \infty, \label{eq_mu_F_2} \\
\int_\Xcal (g_i \vee c) d\nu &= \int_\Xcal (g_i \vee c) d\mu, && i \in \{1,\ldots,d\} \text{ with } \int_\Xcal g_i^- d\mu = \infty. \label{eq_mu_F_3}
\end{align}
Since $\nu$ is finitely supported, all $g_i$ are $\nu$-integrable. We thus have from \eqref{eq_mu_F_2} and \eqref{eq_mu_F_3} that $\nu \in \Pcal$. Then \eqref{eq_mu_F_1} and the fact that $h \in \Ecal$ yield $\int_\Xcal h d\mu \le 1$. Since $\mu \in \widetilde\Pcal$ was arbitrary, this shows that $h \in \widetilde\Ecal$. If $h \in \Ecal$ is unbounded, then by what we just proved, $h \wedge n \in \widetilde\Ecal$ for each $n \in \N$, and thus $h \in \widetilde\Ecal$ by monotone convergence. This completes the proof.
\end{proof}

\section{Summary and open problems} \label{S_summary}

In various applications, statistical hypotheses are naturally described through constraints on the data-generating distribution. In this paper, we use the duality theory for locally convex topological vector spaces, in particular the bipolar theorem, to give an abstract representation of all e-variables for hypotheses of this kind in a general setting. We then show how the abstract representation instantiates in several concrete cases. As such, our results facilitate the use of e-variables for testing hypotheses generated by constraints.

There are, nonetheless, many important and challenging open problems. One of them is to describe the set of e-variables for the i.i.d.\ $n$-sample hypothesis $\Pcal^{\otimes n}$ consisting of all product distributions $\mu^{\otimes n}$ on $\Xcal^n$ with $\mu \in \Pcal$, for some hypothesis $\Pcal$ of distributions on $\Xcal$. Any function of the form $h(x_1,\ldots,x_n) = \prod_{i=1}^n h_i(x_i)$ with $h_i \in \Ecal$ is an e-variable for $\Pcal^{\otimes n}$, as are all convex combinations of such functions, and their limits. However, not all e-variables arise in this way in general. For example, suppose $\Pcal$ is the set of \emph{all} distributions on $\R$, so that $\Pcal^n$ is the set of \emph{all} laws of real i.i.d.\ $n$-samples. Then any $h_i \in \Ecal$ satisfies $h_i \le 1$, and thus so do e-variables of the product form above, their convex combinations, etc. But, taking $n=2$ for illustration, the function $h(x_1,x_2) = 2\bm1_{\{x_1 < x_2\}} + \bm1_{\{x_1 = x_2\}}$ is an example of a nontrivial e-variable for $\Pcal^{\otimes n}$. Can all nontrivial e-variables be characterized?

Another interesting future direction is to extend our results to the sequential setting, where testing and inference based on e-variables have been proven to be particularly effective. Let us mention two specific problems whose solutions would be of interest. We formulate these in a probabilistic language, rather than the analytic language we otherwise use in this paper.

Suppose we observe data $X_t$, $t \in \N$, taking values in a measurable space $\Xcal$. We model $(X_t)_{t \in \N}$ as the canonical process on $\Xcal^\N$, and let $(\Fcal_t)_{t \in \N}$ with $\Fcal_t = \sigma(X_1,\ldots,X_t)$ be the filtration generated by the data. We let $\Fcal_0$ be the trivial $\sigma$-algebra. Consider now a finite constraint set $\Phi = \{g_1,\ldots,g_d\}$ of functions on $\Xcal$, which generates a family $\Pcal$ of distributions on $\Xcal$ as in Definition~\ref{D_constrained_hypotheses}. We may then examine those distributions $\P$ on $\Xcal^\N$ under which the conditional distributions of the data belong to $\Pcal$. That is,
\[
\P(X_t \in \cdot \mid \Fcal_{t-1}) \text{ belongs to } \Pcal \text{ for all } t \in \N, \text{ $\P$-a.s.},
\]
or, equivalently,
\[
\E_\P[g_i(X_t) \mid \Fcal_{t-1}] \le 0, \text{ $\P$-a.s., for all } i = 1,\ldots,d \text{ and all } t \in \N.
\]
Let $\overline\Pcal$ denote the collection of all such $\P$. Note that $\overline\Pcal$ is a family of distributions over the entire data sequence, whereas $\Pcal$ is a family of distributions over a single observation.

An \emph{e-process} for $\overline\Pcal$ is a nonnegative adapted process $(Y_t)_{t \in \N_0}$ (where $\N_0 = \N \cup \{0\}$) such that $\E_\P[Y_\tau] \le 1$ for all $\P \in \overline\Pcal$ and all finite stopping times $\tau$ for the given filtration. E-processes are the natural analog of e-variables in the sequential setting, and it is of interest to describe their structure. One can easily construct a large class of such e-processes by setting $Y_0 = 1$ and
\begin{equation} \label{eq_e_process}
Y_t = \prod_{s=1}^t \left(1 + \sum_{i=1}^d \pi_{i,s} g_i(X_s) \right), \quad t \in \N,
\end{equation}
where $\pi_t = (\pi_{1,t},\ldots,\pi_{d,t})$, $t \in \N$, is a predictable process taking values in the set $\Pi^\Phi$ defined in Corollary~\ref{C_evar_max_finite}. The predictability property means that $\pi_t$ is a measurable function of $X_1,\ldots,X_{t-1}$. Any process $(Y_t)_{t \in \N_0}$ as in \eqref{eq_e_process} above is actually a $\overline\Pcal$-supermartingale, i.e., a supermartingale under every $\P \in \overline\Pcal$. The e-process property then follows from the stopping theorem. It is now natural to ask whether \emph{every} admissible e-process for $\overline\Pcal$ is of the form \eqref{eq_e_process}, at least assuming the constraint qualification \eqref{C_max_evar_CQ}.  \citet{clerico2024optimality} provides a positive answer in the natural sequential extension of Example~\ref{eg:bounded-mean} for testing if a bounded conditional mean equals $m$. It would be of interest to settle the general case as well.

The second open problem concerns the smaller hypothesis $\Pcal^{\infty} \subset \overline\Pcal$ consisting of all i.i.d.\ distributions $\P = \mu^{\otimes\infty}$ on $\Xcal^\N$ with $\mu \in \Pcal$. This hypothesis is significantly smaller than $\overline\Pcal$, so its family of admissible e-processes is potentially larger. Can it be described explicitly? How, if at all, is it related to the family of e-processes for $\overline\Pcal$?

Lastly, one can of course formulate versions of these questions for infinite constraint sets $\Phi$ as well, and this also leads to interesting open problems.

\appendix

\section{Some results from topology and functional analysis}\label{S_functional_analysis}

\subsection{Closedness, compactness, and continuity using nets}\label{S_nets}

Let $X$ be a topological space. In many cases, for example if $X$ is a metric space, properties such as closedness, compactness, and continuity can be characterized using sequences. For example, a set $C \subset X$ is closed if and only if $C$ contains the limit of every convergent sequence $(x_n)_{n \in \N} \subset C$. In general topological spaces, this characterization may fail. However, it can be restored by replacing sequences with the more general concept of \emph{nets}, also known as \emph{Moore--Smith sequences}. We review the basic definitions and facts below, and refer the reader to \citet[Chapter~2]{MR370454}, \citet[Chapter~4]{MR264581}, and \citet[p.~32]{ali_bor_06} for more details.

Instead of only using the natural numbers $\N$ as index set, a net can be indexed by a general \emph{directed set} $A$. This is a nonempty set with a binary relation $\ge$ that is symmetric ($\alpha \ge \alpha$ for all $\alpha \in A$), transitive ($\alpha \ge \beta$ and $\beta \ge \gamma$ implies $\alpha \ge \gamma$ for all $\alpha,\beta,\gamma \in A$), and enjoys the directedness property that for any $\alpha,\beta \in A$ there is $\gamma \in A$ with $\gamma \ge \alpha$ and $\gamma \ge \beta$. The natural numbers with the standard ordering is an example of a directed set; another is the family of all neighborhoods $U$ of a given point $x \in X$, with $U \ge V$ if $U \subset V$. (Recall that a \emph{neighborhood} of a point $x \in X$ is any set that contains an open set containing $x$.) A \emph{net} is a map from some directed set $A$ to $X$, denoted by $(x_\alpha)_{\alpha \in A}$ in analogy with the notation for sequences. For brevity we often write $(x_\alpha)$ or even just $x_\alpha$ for the net $(x_\alpha)_{\alpha \in A}$.
The net \emph{converges} to a point $x \in X$ if it is eventually in any neighborhood of $x$; that is, if for any neighborhood $U$ of $x$, there is some $\alpha \in A$ such that $x_\beta \in U$ for all $\beta \in A$ with $\beta \ge \alpha$. We express this by saying that ``$x_\alpha$ converges to $x$'', or just ``$x_\alpha \to x$''. Lastly, a \emph{subnet} of $(x_\alpha)_{\alpha \in A}$ is a net of the form $(x_{\varphi(\beta)})_{\beta \in B}$, where $\varphi \colon B \to A$ is increasing ($\gamma \ge \beta$ implies $\varphi(\gamma) \ge \varphi(\beta)$) and cofinal (for every $\alpha \in A$ there is $\beta \in B$ such that $\varphi(\beta) \ge \alpha$).

\begin{theorem}\label{T_nets}
\begin{enumerate}
\item\label{T_nets_1} A set $C \subset X$ is closed if and only if it contains all limits of nets in $C$.
\item\label{T_nets_2} A set $C \subset X$ is compact if and only if every net in $C$ has a subnet with a limit in $C$.
\item\label{T_nets_3} A function $f$ from $X$ to a topological space $Y$ is continuous if and only if $x_\alpha \to x$ implies $f(x_\alpha) \to f(x)$. To be precise, the latter property means that for every $x \in X$ and every net $(x_\alpha)_{\alpha \in A}$ in $X$ that converges to $x$, the net $(f(x_\alpha))_{\alpha \in A}$ in $Y$ converges to $f(x)$.
\end{enumerate}
\end{theorem}

\begin{proof}
Parts \ref{T_nets_1} and \ref{T_nets_3} are Theorem~11.7 and Theorem~11.8 of \citet{MR264581}. Part \ref{T_nets_2} follows from Theorem~11.5 and Theorem~17.4 of \citet{MR264581}.
\end{proof}

We now specialize some of the above to Euclidean space $\R^d$. The Heine--Borel theorem states that any closed and bounded subset of $\R^d$ is compact. Therefore, Theorem~\ref{T_nets}\ref{T_nets_2} implies that any bounded net in $\R^d$ has a convergent subnet. We use this fact in the proof of Theorem~\ref{T_finitely_generated}.

For a net $(x_\alpha)_{\alpha \in A}$ in $\R$ one can define the limsup and liminf exactly as for sequences,
\[
\liminf_\alpha x_\alpha = \lim_\alpha \inf_{\beta \ge \alpha} x_\beta
\quad\text{and}\quad
\limsup_\alpha x_\alpha = \lim_\alpha \sup_{\beta \ge \alpha} x_\beta.
\]
That is, $\liminf_\alpha x_\alpha$ is the limit of the net $y_\alpha = \inf_{\beta \ge \alpha} x_\beta$ in the extended reals $[-\infty,\infty]$. This net is increasing in the sense that $\gamma \ge \alpha$ implies $y_\gamma \ge y_\alpha$, and this ensures that the limit exists. We say that $x_\alpha$ converges to infinity if $\liminf_\alpha x_\alpha = \infty$, meaning that $x_\alpha$ is eventually larger than any real number. The case of limsup is analogous. If a net in $\R_+$ does not converge to infinity, meaning that its liminf is finite, it has a bounded subnet. Hence, by Heine--Borel and Theorem~\ref{T_nets}\ref{T_nets_2} as above, it has a further subnet that converges to a limit in $\R_+$. This is again something we make use of in the proof of Theorem~\ref{T_finitely_generated}.

\subsection{Dual pairs and the bipolar theorem}\label{S_dual_pairs}

Here we review some concepts and facts from the classical duality theory of locally convex spaces. All the required material can be found in \cite{sha_wol_99}, see in particular Chapter~IV.

Two real vector space $F$ and $G$ form a \emph{dual pair} (or \emph{dual system}) if there is a bilinear form $\langle \fdot,\fdot \rangle$ on $F \times G$ that separates points in the following sense: if $x \in F$ and $\langle x,y\rangle = 0$ for all $y \in G$, then $x = 0$; and if $y \in F$ and $\langle x,y\rangle = 0$ for all $x \in F$, then $y = 0$. One also says that $\langle\fdot,\fdot\rangle$ places $F$ and $G$ in \emph{(separated) duality}, and writes $\langle F,G\rangle$ as shorthand for the tuple $(F,G,\langle\fdot,\fdot\rangle)$.

Given a dual pair $\langle F,G\rangle$, one defines the \emph{weak topology} $\sigma(F,G)$ as the initial topology generated by the maps $x \mapsto \langle x, y\rangle$, $y \in G$. That is, $\sigma(F,G)$ is the weakest topology on $F$ such that the map $x \mapsto \langle x,y\rangle$ from $F$ to $\R$ is continuous for every $y \in G$. With this topology $F$ is a locally convex space; see \cite{sha_wol_99}, Chapter~II, Section~5.

For any subset $C \subset F$, the \emph{polar} of $C$ (sometimes called the \emph{one-sided polar}) is the set
\[
C^\circ = \{y \in G \colon \langle x, y\rangle \le 1 \text{ for all } x \in C\}.
\]
The \emph{bipolar} of $C$ is the polar of the polar,
\[
C^{\circ\circ} = \{x \in F \colon \langle x, y\rangle \le 1 \text{ for all } y \in C^\circ\}.
\]
The polar and bipolar are always convex. If $C$ is a cone, meaning that $\lambda C \subset C$ for every $\lambda \in [0,\infty)$, then the polar and bipolar are also cones and can be written
\begin{align*}
C^\circ &= \{y \in G \colon \langle x, y\rangle \le 0 \text{ for all } x \in C\}, \\
C^{\circ\circ} &= \{x \in F \colon \langle x, y\rangle \le 0 \text{ for all } y \in C^\circ\}.
\end{align*}
We make extensive use of the following fundamental result, which we state here for convex cones. This follows from Theorem~1.5 of Section~IV in \cite{sha_wol_99} along with the fact that the convex hull of $\{0\} \cup C$ is just $C$ itself when $C$ is a convex cone.

\begin{theorem}[Bipolar theorem] \label{T_bipolar_theorem}
If $C \subset F$ is a convex cone, then the bipolar $C^{\circ\circ}$ is the $\sigma(F,G)$-closure of $C$.
\end{theorem}

\subsection{The Krein--\v{S}mulian theorem} \label{S_krein_smulian}

Let $E$ be a Banach space with norm $\|\cdot\|$. Its dual space $E'$ consists of all bounded linear functionals on $E$, and is equipped with the dual norm given by $\|\varphi\|' = \sup \{ \varphi(x) \colon x \in E, \ \|x\| \le 1\}$. The weak$^*$ topology on $E'$ is the initial topology generated by the maps $\varphi \mapsto \varphi(x)$ from $E'$ to $\R$, where $x \in E$. This is also the topology $\sigma(E',E)$ coming from the dual pair $\langle E',E\rangle$ with bilinear form $\langle \varphi, x \rangle = \varphi(x)$; see Section~\ref{S_dual_pairs}. The following result plays a crucial role in Section~\ref{S_sub_psi}. For a proof, see Theorem~12.1 in \cite{con_90}.

\begin{theorem}[Krein--\v{S}mulian] \label{T_Krein_Smulian}
Let $(E,\|\fdot\|)$ be a Banach space with dual space $(E',\|\fdot\|')$. A convex subset $C \subset E'$ is weak$^*$ closed if and only its intersection with every dual ball is weak$^*$ closed, that is, $C \cap \{\varphi \in E' \colon \|\varphi\|' \le r\}$ is weak$^*$ closed for all $r \in (0,\infty)$.
\end{theorem}

\subsection*{Acknowledgments}
AR acknowledges support from the Sloan Fellowship and grant NSF DMS-2310718.
ML acknowledges support by the National Science Foundation under grant
NSF DMS-2510965.

\bibliography{main}
\bibliographystyle{plainnat}

\end{document}